\documentclass[11pt, reqno]{amsart}

\usepackage{amsfonts,amssymb,amsmath,amsthm, ulem}
\usepackage[headinclude,DIV13]{typearea}
\usepackage[utf8]{inputenc}
\usepackage{hyperref}
\usepackage{geometry}
\usepackage{graphicx, psfrag,enumerate,enumitem}

\usepackage{soul}

\usepackage{color}
 \usepackage{setspace}

\newcommand{\tX}{\widetilde{X}}

\newcommand{\tT}{\widetilde{T}}

\newcommand{\ttau}{\widetilde{\tau}}

\newtheorem{thm}{Theorem}[section]
\newtheorem{prop}[thm]{Proposition}

\newtheorem{cor}[thm]{Corollary}
\newtheorem{lemma}[thm]{Lemma}

\newtheorem{ass}{Assumption}

\newtheorem{rem}[thm]{Remark}
\numberwithin{equation}{section}

\theoremstyle{definition} 
\theoremstyle{definition}

\renewcommand{\P}{\mathbb{P}}
\newcommand{\R}{\mathbb{R}}

\newcommand{\E}{\mathbb{E}}

\newcommand{\N}{\mathbb{N}}

\newcommand{\eps}{\varepsilon}

\numberwithin{equation}{section}

\begin{document}

\title[Long time behaviour of branching processes with catastrophes]
{Long time behaviour of continuous-state nonlinear branching processes with catastrophes}

\author{Aline Marguet}
\address{Aline Marguet, Univ. Grenoble Alpes, INRIA, 38000 Grenoble, France}
\email{aline.marguet@inria.fr}

\author{Charline Smadi}
\address{Charline Smadi, Univ. Grenoble Alpes, INRAE, LESSEM, 38000 Grenoble, France
 and Univ. Grenoble Alpes, CNRS, Institut Fourier, 38000 Grenoble, France}
\email{charline.smadi@inrae.fr}

\date{}

\maketitle

\begin{abstract}
Motivated by the study of a parasite infection in a cell line, we introduce a general class of Markov processes for the modelling of population dynamics.
The population process evolves as a diffusion with positive jumps whose rate is a function of the population size. It also undergoes catastrophic events which kill a fraction of the population, at a rate depending on the population state. 
We study the long time behaviour of this class of processes.
 \end{abstract}

 \vspace{0.2in}


\bigskip




\section{Introduction}
We introduce a general class of non-negative continuous-time and space Markov processes, including diffusive terms, as well as negative and positive jumps. 
They can be seen as a generalization of a class of continuous-state nonlinear branching processes introduced recently in \cite{li2017general}, which did not allow for negative jumps.
Our motivation comes from the study of a parasite infection in a cell population (see the companion paper \cite{long}). The processes studied in the current work may indeed be interpreted as the dynamics of the quantity of parasites in a cell line. Catastrophic events correspond to cell divisions during which a cell splits its parasites between its two daughter cells, according to a probability kernel $\kappa (d\theta)$ on $(0,1)$. The quantity of parasites in a cell line is thus multiplied by $\theta \in (0,1)$, which can also be interpreted as the death of a fraction $(1-\theta)$ of the parasites in the cell line.
First, we investigate the possibility for the processes to be absorbed or to explode in finite time. In each case, we give conditions under which the event has null or positive probability, and even provide conditions under which it happens almost surely. Building on these results, we then explore the long time behaviour of the process.
We give criteria for the processes to converge to a positive random variable,
to $0$ or to $\infty$.
Moreover, in the case of almost sure extinction, we give bounds on the exponential decay of the survival probability.

The class of processes under study belongs to a class of processes recently introduced as 
strong solutions of Stochastic Differential Equations (SDE) in \cite{palau2018branching}, and only few processes of this class have been studied until now.
This framework allows to take into account interactions between individuals as well as the 
effects of the environment. 
The addition of interactions between individuals in continuous-time branching processes has recently attracted a lot of interest. For instance, 
Feller diffusions and Continous State Branching Processes (CSBPs) with logistic competition have been studied in \cite{lambert2005branching,le2013trees} and 
\cite{berestycki2018ray}, respectively, Feller diffusions with some nonlinear birth rates have been studied in \cite{pardoux2014path}, 
and polynomial interactions have been considered in \cite{li2018continuous}.
Li and coauthors \cite{li2017general} have recently introduced a general class of continuous state nonlinear branching processes, 
and have investigated extinction, explosion and coming down from infinity for this class.
However, in all these models, only positive jumps are allowed. They result from large birth events, 
and were first introduced by
constructing the continuous state process 
as the limit of a sequence of discrete branching processes (see for instance 
\cite{lamperti1967limit,bansaye2015scaling}).
In parallel, models where the interactions between individuals result from the fact that the whole population is subject to the variations of the same environment
have been intensively studied recently, in particular in the framework of CSBPs in random environment.
This class of models, initially introduced by Keiding and Kurtz \cite{kurtz1978diffusion,keiding1975extinction} in the case of Feller diffusions in a Brownian environment, 
have been generalised and studied by many authors during the last decade
\cite{boeinghoff2011branching,BPS,palau2016asymptotic,
palau2017continuous,he2018continuous,palau2018branching,
li2018asymptotic,bansaye2019extinction}.
In this setting, negative jumps may occur, being for instance the result of environmental catastrophes killing each 
individual with the 
same probability \cite{BPS}. However, in CSBPs in random environment, the environment is independent of the population state. In particular the 
rate of catastrophes does not depend on the population size. We relax this assumption in the current work.\\

In the next section, we define the processes of interest and give sufficient conditions for their existence and uniqueness as the solution of an SDE. Sections \ref{section_abs_lineage}, \ref{section_expl_lineage} and \ref{sec:condH} are dedicated to the possibility of absorption and explosion of the process in finite time. In Section \ref{section_class} we study the long time behaviour of the process. The proofs are derived in Section \ref{section_proofs}.\\

In the sequel, we work on a filtered probability space $(\Omega, \mathcal{F}, \mathcal{F}_t,\mathbb{P})$, $\N:=\{0,1,2,...\}$ will denote the set of non-negative integers, $\R_+:=[0,\infty)$ the real line 
and $\R_+^*:=(0,\infty)$.
We will denote by $\mathcal{C}_b^2(\mathbb{R}_+)$ the set of twice continuously differentiable bounded functions on $\mathbb{R}_+$. Finally, for any stochastic 
process $X$ on $\mathbb{R}_+$, we will denote by 
$\mathbb{E}_x\left[f(X_t)\right]=\E\left[f(X_t)\big|X_0 = x\right]$.

\section{Definition of the population process} \label{section_model}

We consider continuous-time and continuous-state Markov processes solution to the following SDE: 
\begin{align}\label{eq:EDS}
X_t= X_0+ \int_0^t g(X_s) ds+ \int_0^t\sqrt{2 \sigma^2(X_s) }dB_s+ \int_0^t\int_0^{p(X_{s^-})}\int_{\mathbb{R}_+}z\widetilde{Q}(ds,dx,dz)&\\
+ \int_0^t  \int_0^{r(X_{s^-})} \int_0^1 (\theta-1)X_{s^-}N(ds,dx,d\theta)&, \nonumber
\end{align}
where $X_0$ is non-negative, $g$ is a real function on $\R_+$, $\sigma$, $p$ and $r$ are non-negative functions on $\R_+$, $B$ is a standard Brownian motion,  
$\widetilde{Q}$ is a compensated Poisson point measure with intensity 
$ds\otimes dx\otimes \pi(dz)$, $\pi$ is a positive measure on $\mathbb{R}_+$, and $N$ is a Poisson point measure with intensity 
$ds\otimes dx\otimes  \kappa(d\theta)$ where $\kappa$ is a probability measure on $[0,1]$. We assume that $N$, $\tilde{Q}$ and $B$ are mutually independent. 

Under some mild conditions, the SDE \eqref{eq:EDS} has a unique pathwise strong solution.
We will work under these conditions in the sequel.

\begin{ass}\label{ass_PSS}
\
\begin{itemize}
\item[-]The functions $r$ and $p$ are locally Lipschitz and $p(0) = 0$, $r(0)<\infty$.
\item[-]The function $g$ is continuous on $\mathbb{R}_+$, $g(0)=0$ and for any $n \in \N$ there exists a finite constant $B_n$ such that for any $0 \leq x \leq y \leq n$
\begin{align*} |g(y)-g(x)|
\leq B_n \phi(y-x),
\end{align*}
where
\begin{align*}
\phi(x) = \left\lbrace\begin{array}{ll}
 x \left(1-\ln x\right) & \textrm{if } x\leq 1,\\
 1 & \textrm{if } x>1.\\
\end{array}\right.
\end{align*}
\item[-]The function $p$ is non-decreasing on $\mathbb{R}_+$.
\item[-]The function $\sigma$ is H\"older continuous with index $1/2$ on compact sets and $\sigma(0)=0$. 
\item[-]The measure $\pi$ satisfies
$$ \int_0^\infty \left(z \wedge z^2\right)\pi(dz)<\infty. $$
\end{itemize}
\end{ass}

The form of this assumption comes from the conditions of \cite[Proposition 1]{palau2018branching} that we will apply to get the next result. The condition $g(0)=p(0)=\sigma(0)=0$ 
ensures that the process stays non-negative and that $0$ is an absorbing state.
Notice that the second point makes sense from a biological point of view. The value of $p(x)$ corresponds to the rate of large reproductive events when the population is of size $x$. It thus means that more individuals produce more offspring.

\begin{prop} \label{prop_sol_SDE}
Suppose that Assumption \ref{ass_PSS} holds.
 Then, Equation \eqref{eq:EDS} has a pathwise unique non-negative strong solution absorbed at $0$ and $\infty$.
 It is a Markov process with infinitesimal generator $\mathcal{G}$, satisfying for all $f\in C_b^2(\mathbb{R}_+)$,
\begin{align}
\mathcal{G}f(x) = g(x)f'(x)+\sigma^2(x)f''(x)+p(x)\int_{\mathbb{R}_+}\left(f(x+z)-f(x)-zf'(x)\right)\pi(dz)& \nonumber \\
 + r(x)\int_0^1\left(f(\theta x)-f(x)\right)\kappa(d\theta)&. \label{inf_gene}
\end{align}
\end{prop}

Here, we adopt the formalism of \cite{palau2018branching} for the definition of a solution to the SDE \eqref{eq:EDS}: a $[0,\infty]$-valued process $X=(X_t,t \geq 0)$ is a solution if it satisfies \eqref{eq:EDS} up to the time $\tau_n:= \inf \{ t \geq 0, X_t\geq n \}$ for all $n \geq 1$, and $X_t=\infty$ for all $t \geq \tau := \lim_{n \to \infty} \tau_n$.\\

We can now study the long time behaviour of the process $X$ solution to \eqref{eq:EDS}.

\section{Absorption of the process} \label{section_abs_lineage}

A first question, which is natural when modelling populations, is to know if the process can get extinct in finite time.
Let us introduce the stopping times $\tau^-(x)$ and $\tau^+(x)$ via
\begin{equation} \label{tau1}
\tau^-(x) := \inf\lbrace t\geq 0 : X_t<x\rbrace, \quad \tau^+(x): = \inf\lbrace t\geq 0 : X_t>x\rbrace, \quad \text{for }x>0
\end{equation}
and
\begin{equation} \label{tau2}
\tau^-(0):= \inf\lbrace t\geq 0 : X_t = 0\rbrace,
\end{equation}
with the convention $\inf\emptyset: = \infty$. Moreover, we denote by $\Theta$ a random variable distributed according to $\kappa$.\\

The study of the absorption of the process $X$ relies on the construction of a sequence of martingales. 
Let us define
\begin{align*}
\mathcal{A} := \left\lbrace a>1,\ \E[\Theta^{1-a}]<\infty\right\rbrace
\end{align*}
and the set of functions $G_a$ given for $a \in \mathcal{A} \cup (0,1)$ and $x>0$ by
\begin{align}\label{eq:Ga}
G_a(x) : =  (a-1)\left(\frac{g(x)}{x}-a\frac{\sigma^2(x)}{x^2}-r(x) \frac{1-\E[\Theta^{1-a}]}{1-a}-p(x)I_a(x)\right),
\end{align}
where 
\begin{align}\label{def:Ia}
I_a(x)= ax^{-2}\int_{\mathbb{R}_+} z^2\left(\int_0^1 (1+zx^{-1}v)^{-1-a}(1-v)dv\right)\pi(dz).
\end{align}

The behaviour around $0$ of $G_a$ characterizes the likelihood of the process $X$ to reach $0$.
It depends on everything except the negative jump term close to $0$.
This is because the division rate $r$ is bounded for small values of the process ($r(0)<\infty$ and $r$ is continuous), and thus the process cannot reach $0$ due to an accumulation of jumps.
We entitle the conditions \textbf{SN0} (Small Noise around $0$) and \textbf{LN0} (Large Noise around $0$).
\begin{enumerate}[label=\bf{(SN0)}]
\item \label{A1} There exists $a\in\mathcal{A}$ and a non-negative function $f$ on $\mathbb{R}_+$ such that 
\begin{equation} \label{condi_noext} \frac{g(x)}{x}-a\frac{\sigma^2(x)}{x^2}-p(x)I_a(x)= f(x) + o(\ln x),\quad (x\rightarrow 0).
\end{equation}
\end{enumerate}
\begin{enumerate}[label=\bf{(LN0)}]
\item\label{A2}  There exist $a<1$, $\eta>0$ and $x_0> 0$ such that for all $x\leq x_0$
\begin{equation}\label{condi_ext}
\frac{g(x)}{x}-a\frac{\sigma^2(x)}{x^2}-p(x)I_a(x) \leq - \left(\ln(x^{-1})\right) \left(\ln(\ln(x^{-1}))\right)^{1+\eta}.
\end{equation}
\end{enumerate}
\begin{rem}
Condition \ref{A2} can be generalized. Indeed, a careful reading of the proof shows that
a sufficient condition is that there exists a positive non-increasing function $\mathfrak{f}$ on $\mathbb{R}_+$ such that
\begin{enumerate}
\item[i)] there exist $a<1$ and $x_0>0$ such that for all $x\leq x_0$,
\begin{align*}
\frac{g(x)}{x}-a\frac{\sigma^2(x)}{x^2}-p(x)I_a(x) \leq -\mathfrak{f}(x),
\end{align*}
\item[ii)] there exist $\eps<1$ and $\delta>0$ such that
\begin{align*}
\sum_{n=1}^\infty (1+\delta)^n\mathfrak{f}\left(\eps^{(1-\delta)(1+\delta)^n}\right)^{-1}<\infty.
\end{align*}
\end{enumerate} 
\end{rem}

Under a first moment assumption for the positive jumps, Conditions {\bf{(SN0)}} and {\bf{(LN0)}}
may be simplified.

\begin{rem} \label{rem_simpli}
If $\int_{\R_+}z\pi(dz)<\infty$, \eqref{condi_noext} is equivalent to 
\begin{equation*}\frac{g(x)}{x}-a\frac{\sigma^2(x)}{x^2}=f(x)+ o(\ln x ),\quad (x\rightarrow 0),
\end{equation*}
and \eqref{condi_ext} is equivalent to 
\begin{equation*}
\frac{g(x)}{x}-a\frac{\sigma^2(x)}{x^2} \leq -\left(\ln(x^{-1})\right) \left(\ln(\ln(x^{-1}))\right)^{1+\eta}.
\end{equation*}
See Section \ref{sec:proofsec3} for the proof of this result. 
\end{rem}

We can now state results on the absorption of the process in terms of those two conditions.
\begin{thm}\label{thm:absorption}
Suppose that Assumption \ref{ass_PSS} holds and let $X$ be the pathwise unique solution to \eqref{eq:EDS}.
\begin{itemize}
\item[i)] If Condition \ref{A1} holds,
then $\mathbb{P}_x\left(\tau^-(0)<\infty\right)=0$ for all $x>0$. 
\item[ii)] If Condition \ref{A2} holds, then $\mathbb{P}_x\left(\tau^-(0)<\infty\right) >0$ for all small enough $x>0$.
\item[iii)]
If Condition \ref{A2} holds and if $r(x)>0$ for every $x>0$, then for any $x>0$, $\mathbb{P}_x\left(\tau^-(0)<\infty\right) >0$.
\end{itemize}
\end{thm}

Theorem \ref{thm:absorption} extends \cite[Theorem 2.3]{li2017general}, 
and we thus use some ideas of the proof of \cite[Theorem 2.3]{li2017general} to derive our results.
Moreover, several adaptations are needed as negative jumps may occur in our process.
Notice that we give tighter bounds than in \cite[Theorem 2.3]{li2017general} where $(\ln(x^{-1}))^{r}, r<1$ and 
$(\ln(x^{-1}))^{r}, r>1$ were considered instead of the right-hand sides of \eqref{condi_noext} and \eqref{condi_ext}.\\

Theorem \ref{thm:absorption} shows that the behaviour of the noise around zero determines the fate of the process in terms of absorption: if it is large enough compared to the growth rate of the parasites around zero, then the probability of being absorbed in finite time is positive.
Notice that under our conditions, the process cannot be absorbed because of negative jumps.

\section{Explosion of the process} \label{section_expl_lineage}

We now focus on the possibility for the process to explode in finite time. 
For the modelling of a parasite infection, it would correspond to an explosion of the quantity of parasites in a lineage or in the cell population.
Unable to overcome the infection, the latter would thus be likely to die (see companion paper \cite{long}).\\

We define 
$$ \tau^+(\infty):=\lim_{n \to \infty} \tau^+(n), $$
and we are interested in the probability of the event $\{\tau^+(\infty)<\infty\}$.

In this case, the behaviour of $G_a$ at infinity determines the likelihood of the process to reach infinity in finite time. Unlike for the absorption behaviour, the law and frequency of negative jumps may impact the probability of explosion of the process.
We entitle the conditions \textbf{SN$\infty$} (Small Noise for large values) and \textbf{LN$\infty$} (Large Noise for large values).
\begin{enumerate}[label=\bf{(SN$\infty$)}]
\item \label{SNinfty} 
There exist $a<1$ and a non-negative function $f$ on $\mathbb{R}_+$ such that
\begin{align*} 
 \frac{g(x)}{x}-a\frac{\sigma^2(x)}{x^2}-r(x)\frac{1-\E[\Theta^{1-a}]}{1-a} -p(x)I_a(x)=-f(x)+ o(\ln x),\quad (x\rightarrow +\infty).
\end{align*}
\end{enumerate}
\begin{enumerate}[label=\bf{(LN$\infty$)}]
\item\label{LNinfty}  There exist $a\in\mathcal{A}$, $\eta>0$ and $x_0> 0$ such that for all $x\geq x_0$
\begin{align*}
\frac{g(x)}{x}-a\frac{\sigma^2(x)}{x^2}-r(x) \frac{1-\E[\Theta^{1-a}]}{1-a} -p(x)I_a(x) \geq \ln x \left(\ln(\ln x)\right)^{1+\eta}.
\end{align*}
\end{enumerate}

The next result describes the possible behaviours for the process in terms of explosion.

\begin{thm}\label{thm:explosion}
Suppose that Assumption \ref{ass_PSS} holds and let $X$ be the pathwise unique solution to \eqref{eq:EDS}.
\begin{itemize}
\item[i)] If Condition \ref{SNinfty} holds, then $\mathbb{P}_x\left(\tau^+(\infty)<\infty\right)=0$ for all $x>0$. 
\item[ii)] If Condition \ref{LNinfty} holds then $\mathbb{P}_x\left(\tau^+(\infty)<\infty\right) >0$ for all large enough $x>0$.
\item[iii)] If Condition \ref{LNinfty} holds and if $\sigma(x)+p(x)>0$ for every $x>0$, then for any $x> 0$, $\P_x(\tau^+(\infty)<\infty)>0$ .
\end{itemize}
\end{thm}

Theorem \ref{thm:explosion} extends \cite[Theorem 2.8]{li2017general}, and
again we use some ideas of this previous work.
However, several adaptations are again needed as negative jumps may occur in our process. Moreover, we completed the proof of 
\cite[Theorem 2.8]{li2017general} as one argument seemed to be missing to conclude.
Finally, as in Theorem \ref{thm:absorption}, we give tighter bounds than in \cite[Theorem 2.8]{li2017general}.\\

It is interesting to notice that the explosion of the process depends on all the components of the population dynamics. 
The Malthusian growth rate $g$ increases the likelihood of the explosion phenomenon when increasing, whereas the fluctuations of the Brownian part and due to the 
large reproductive events decrease it. An interesting consequence of this result is that the presence of the catastrophic events may prevent the explosion of the population. In particular, if the process represents the quantity of parasites in a cell line and the catastrophes correspond to the sharing of 
the parasites between the two daughter cells at division, the cell population may avoid the explosion of the quantity of parasites by increasing its division rate 
or by modifying the law of the sharing of the parasites between the two daughter cells.

\section{Simpler conditions for absorption or explosion of the process}\label{sec:condH}

When the fluctuations of the process $X$ are no too strong (see Proposition \ref{prop_ext_cell_line} for details), the conditions for absorption and explosion of the process $X$ take simpler expressions. In particular, they do not rely on the existence of a positive real number $a$ satisfying some conditions. As we will see, they allow to make links with previous results on CSBP in random environment.

From now on, we will always assume that the following conditions hold:
\begin{equation}\label{cond_lntheta}
 \left|\E[\ln \Theta]\right|=\left|\int_0^1 \ln \theta\kappa(d\theta)\right|<\infty,\quad \int_{\R_+}\ln(1+z)\pi(dz)<\infty. 
\end{equation}
We introduce a new function $H$ (linked to the family $(G_a, a \neq 1)$) whose behaviour at $0$ (resp. infinity) is linked to the absorption (resp. explosion) behaviour of the process $X$:
\begin{equation*}\label{def:H}
H(x) := \lim_{a \to 1} \frac{G_a(x)}{a-1}= \frac{g(x)}{x}-\frac{\sigma^2(x)}{x^2}+r(x)\E\left[\ln \Theta \right] -p(x)I(x),
\end{equation*}
where $G_a$ is defined in \eqref{eq:Ga}, 
\begin{equation} \label{def:I} I(x):=\lim_{a\rightarrow 1}I_a(x) = -\int_{\mathbb{R}_+} \left[\ln\left(1+zx^{-1}\right)-zx^{-1}\right]\pi(dz),\end{equation}
and $I_a$ is defined in \eqref{def:Ia}. We refer the reader to Appendix \ref{app:I} for the derivation of the limit. Note that $I$ is well-defined under the classical moment assumption $\int_{\R_+} z\wedge z^2\pi(dz)<\infty$.
Using Theorems \ref{thm:absorption} and \ref{thm:explosion} we can prove the following result.

\begin{prop}\label{prop_ext_cell_line}
Let $X$ be the pathwise unique solution to \eqref{eq:EDS}, suppose $\int_{\R_+}z^2\pi(dz)<\infty$.
\begin{enumerate}
\item[\bf{(Absorption)}]
\begin{itemize}
\item[i)] If
$$ H(x)=o(\ln x)\quad \text{and}\quad \frac{\sigma^2(x)}{x^2} +\frac{p(x)}{x^2}=\mathcal{O}(\ln x),  \ (x\rightarrow 0)$$
then for all $x>0$, $\mathbb{P}_x\left(\tau^-(0)<\infty\right)=0.$
\item[ii)] If there exist $\eta>0$ and $x_0>0$ such that $\forall x<x_0$,
\begin{align*}
H(x)\leq -\ln(x^{-1})\left(\ln\ln (x^{-1})\right)^{1+\eta}\quad  \text{and} \quad
\frac{\sigma^2(x)}{x^2} +\frac{p(x)}{x^2}=\mathcal{O}(\ln (x^{-1}) \left(\ln\ln (x^{-1})\right)^{1+\eta}),\quad (x\rightarrow 0),
\end{align*}
then for all $x>0$, $\mathbb{P}_x\left(\tau^-(0)<\infty\right)>0.$
\end{itemize}
\item[\bf{(Explosion)}]
\begin{itemize}
\item[i)]If
$$ H(x)=o(\ln x)\quad \text{and}\quad \frac{\sigma^2(x)}{x^2} + r(x) +\frac{p(x)}{x^2}=\mathcal{O}(\ln x),  \ (x\rightarrow +\infty)$$
then for all $x>0$, $\mathbb{P}_x\left(\tau^+(\infty)<\infty\right)=0.$
\item[ii)] If there exist $\eta>0$ and $x_0>0$ such that $\forall x>x_0$,
\begin{align*}
H(x)\geq \ln x\left(\ln\ln x\right)^{1+\eta}\quad \text{and}\quad 
\frac{\sigma^2(x)}{x^2} + r(x) +\frac{p(x)}{x^2}=\mathcal{O}(\ln x \left(\ln\ln x\right)^{1+\eta}),\quad (x\rightarrow +\infty),
\end{align*}
then for all $x>0$, $\mathbb{P}_x\left(\tau^+(\infty)<\infty\right)>0.$
\end{itemize} 
\end{enumerate}
\end{prop}

We thus see that for the phenomena of absorption and explosion, the trade-off between the growth of the parasites and the division of the parasites between the 
two daughter cells is fully described by the behaviour of the function $H$ at zero and infinity. In fact, and as we will see in the next section, 
the long time behaviour of the infection in a cell line is governed by the behaviour of this function. However, to conclude on the behaviour of the process in finite time, we also need the variance of the noise and the rate of positive jumps to be small enough around $0$ or infinity.
Note that if $g(x)\equiv gx, r(x)\equiv r$ and $p(x)\equiv 0$, we have $H(x) = g+\E[\ln\Theta]-\sigma(x)^2x^{-2}$ and we retrieve the key quantity $g+\E[\ln\Theta]$ found in \cite{BT11} but unfortunately, the condition on $\sigma$ for the case of absorption is too strong to be satisfied by the standard noise of Feller diffusions $\sigma(x)^2 =\sigma ^2 x$. However, the weaker assumption \ref{A2} is satisfied in this case. But for the finite time behaviour of the process, rather than the sign of $g+\E[\ln\Theta]$, it is the strength of the fluctuations that matters.

More generally when $X$ is a CSBP in random environment, there exists a Lévy process $\bar{K}$ such that $(X_te^{-\bar{K}_t}, t \geq 0)$ is a non-negative local martingale, and as such a non-negative supermartingale which converges to a non-degenerate random variable (see \cite{BPS,palau2016asymptotic,li2018asymptotic,he2018continuous,bansaye2019extinction} for instance). The expectation of $\bar{K}_1$ thus gives information on the long time behaviour of the process $X$ in this case. Our function $H$ is in fact an extension of this expectation in the non-linear case. This link will be more explicit in the next section.

\section{Long time behaviour of the process} \label{section_class}

The long time behaviour of the process $X$ depends on the interplay between $g$, which tends to increase (resp. decrease) it when positive (resp. negative), $r$, which decreases it, and the fragmentation kernel $\kappa$ which has a less intuitive effect. 
It is also impacted by the random fluctuations of the large birth events.
We consider the following possibilities for the relative strengths of $g$ and $r$ (Local Slow/Fast Growth (LSG, LFG), Global Slow/Fast/Very Fast Growth (GSG, GFG, GVFG):
\begin{enumerate}[label =\bf{(LSG)}]
\item \label{B0} There exist $\eta>0$ and $x_0\geq 0$ such that 
\begin{equation*} 
 H(x) \leq - \eta, \quad \forall \ x > x_0.
\end{equation*}
\end{enumerate}
\begin{enumerate}[label =\bf{(LFG)}]
\item \label{LFG} There exist $\eta>0$ and $x_1\geq 0$ such that 
\begin{equation*} 
 H(x) \geq  \eta, \quad \forall \ x < x_1.
\end{equation*}
\end{enumerate}
\begin{enumerate}[label =\bf{(GSG)}]
\item \label{B1} There exist $\underline{r}>0$ and $\eta\geq 0$ 
such that $r(x) \geq \underline{r}$, $\forall \ x \geq 0$ and
\begin{equation*} 
 \frac{g(x)}{xr(x)} + \E\left[\ln \Theta \right] \leq - \eta, \quad \forall \ x > 0.
\end{equation*}
\end{enumerate}
\begin{enumerate}[label =\bf{(GFG)}]
\item \label{B2} There exist $\underline{r}>0$ and $\eta> 0$ such that
 $r(x) \geq \underline{r}$, $\forall \ x \geq 0$ and
\begin{equation*} 
  \frac{g(x)}{xr(x)} + \E\left[\ln \Theta \right] \geq  \eta, \quad \forall \ x > 0.
\end{equation*}
\end{enumerate}
\begin{enumerate}[label =\bf{(GVFG)}]

\item \label{B4} There exist $\underline{r}>0$ and $\eta\geq  0$ such that
 $r(x) \geq \underline{r}$, $\forall \ x \geq 0$ and
\begin{equation*} 
  \frac{g(x)}{xr(x)} + \E\left[\ln \Theta \right]- \frac{2\sigma^2(x)}{x^2r(x)}-
  \frac{p(x)}{r(x)}\int_{\R_+}\frac{z^2x^{-2}}{1+zx^{-1}}\pi(dz) \geq  \eta, \quad \forall \ x > 0.
\end{equation*}
\end{enumerate}

\begin{rem}\label{rem:cond-slowfast}
Let us make some remarks on these conditions
\begin{itemize}
\item[$\bullet$] Condition \ref{B0} is satisfied in particular if there exist $\eta,x_0>0$ such that
$$
 \frac{g(x)}{x} + r(x) \E\left[\ln \Theta \right] \leq - \eta, \quad \forall \ x > x_0,
$$
as $\ln(x+z)-\ln x-z/x\leq 0$ for all $x,z> 0$ by the Mean Value Theorem.
\item[$\bullet$] \ref{B4} implies \ref{B2}.
\item[$\bullet$] \ref{B4} implies \ref{A1}.
This follows from the fact that if $\mathcal{A}$ is non-empty, we can find $a \in \mathcal{A}$ 
such that the following inequality holds (see the proof on page \pageref{no_abs}):
\begin{align}
\frac{g(x)}{x}-a\frac{\sigma^2(x)}{x^2}-p(x)I_a(x)\geq 
 \frac{g(x)}{x} - \frac{2\sigma^2(x)}{x^2}-
  p(x)\int_{\R_+} \frac{z^2x^{-2}}{1+zx^{-1}}\pi(dz) .
  \label{rem_no_abs}
\end{align}
In particular, the process $X$ cannot reach $0$ under Assumption \ref{B4}.
\end{itemize}
\end{rem}

The next result states in particular that under Condition \ref{B0}, the division mechanism and the random fluctuations overcome the growth of $X$.
In this case, the process $X$ converges 
to a finite variable, which may be $0$ if $X$ can be absorbed.

\begin{thm}\label{th:convXt}
Suppose that Assumptions \ref{ass_PSS} holds.
\begin{itemize}
\item If Conditions \ref{A1}, \ref{SNinfty} hold and \ref{B0} or \ref{LFG} is satisfied, then, for all $x\geq 0$,
the process $(X_t,t\geq 0)$ converges in law as $t$ tends to infinity to $X_{\infty}$ satisfying
\begin{equation}\label{esp_Xinfty}
\E_x \left[g(X_\infty)- X_\infty r(X_\infty)\left(1-\E\left[\Theta\right]\right) \right]=0.
\end{equation}
Moreover, the distribution of $X_\infty$ is the unique stationary distribution of the process $X$ and for every bounded and measurable function $f$, 
almost surely, 
$$ \lim_{t \to \infty} \frac{1}{t} \int_0^t f(X_s)ds= \E[f(X_\infty)]. $$
\item If Condition \ref{SNinfty} holds, if there exist $\eps,x_0>0$ such that \ref{A2} holds for $x \leq \eps$ and \ref{B0} holds for $x \geq x_0$, and if $r>0$ on $[\eps \wedge e^{-1},x_0]$
then for all $x\geq 0$,
\begin{equation} \label{extin_B0A2}
\mathbb{P}_x\left(\exists t < \infty, X_t = 0\right) = 1. 
\end{equation}
\item If Condition \ref{A1} holds, if there exist $\eps,x_0>0$ such that \ref{LNinfty} holds for $x \geq 1/\eps$ and \ref{LFG} holds for $x \leq x_0$, and if $p>0$ on $[x_0,1/\eps]$,
then for all $x\geq 0$,
\begin{equation} \label{expl_A1LFG}
\mathbb{P}_x\left(\exists t < \infty, X_t = \infty\right) = 1. 
\end{equation}

\end{itemize}

\end{thm}
The second point of this result generalizes \cite[Proposition 1.1]{BT11} to the case of more general parasites dynamics. Indeed, in \cite{BT11}, the authors considered the case $g(x) = g x$, $\sigma(x)^2 = \sigma^2 x$ and $p(x)\equiv 0$ for some $g,\sigma>0$. Note that in this case, if $r$ is constant, \ref{A2} and \ref{SNinfty} always hold and \ref{B0} reduces to $g+r\E[\ln \Theta]<0$, which is the condition stated in \cite[Proposition 1.1 i)]{BT11}. If $r$ is a non-increasing or non-decreasing function, we also retrieve the same conditions as in \cite[Proposition 1.1 ii)iii)]{BT11}. The function $H$ describes the strengths of the different mechanisms, so that Conditions \ref{B0} and \ref{LFG} determine the fate of the infection, depending on which mechanism overcomes the others at critical parasites concentrations (small or large).\\

Finally, we provide some properties on the long time behaviour of $X$ under Assumptions \ref{B1}, \ref{B2} and \ref{B4}, 
extending the classification for stable CSBPs with random catastrophes (corresponding to $r(x)=r$, $g(x)=gx$, $\sigma(x)=\sigma \sqrt{x}$ and $\pi \equiv 0$ or 
$\sigma \equiv 0$, $p(x)=x$ and $\pi$ stable). The first result extends \cite[Corollary 2]{BPS}.

\begin{prop} \label{3_behaviours}
Suppose that Assumption \ref{ass_PSS} is satisfied.
\begin{itemize}
 \item[i)] If Condition \ref{B1} holds for $\eta>0$, then 
 $$ \lim_{t \to \infty} X_t = 0,  \quad \text{almost surely}.$$
\item[ii)] If Condition \ref{B1} holds for $\eta=0$, then 
 $$ \liminf_{t \to \infty} X_t = 0,  \quad \text{almost surely}.$$
 \item[iii)] If Condition \ref{B2} holds, if there exists $\eps>0$ such that
 $$ \int_0^\infty z \ln^{1+\eps}(1+z) \pi(dz)<\infty $$
 and if the function $x \mapsto (\sigma^2(x)+p(x))/x$ is bounded, then 
$$ \P_x\left(\liminf_{t \to \infty} X_t>0\right)>0. $$
\end{itemize}
\end{prop}

In the last case, we additionally prove in the next corollary that with positive probability, $X$ grows (at least) exponentially.
Moreover, when the diffusion term is large enough ($\sigma(x)$ larger than $\sqrt{x}$, which corresponds to Feller diffusion), we are 
able to provide a bound on the absorption rate in the two first cases.

\begin{cor}\label{iiimoreprecise}
Suppose that Assumption \ref{ass_PSS} is satisfied.
\begin{itemize}
 \item[i)] If Condition \ref{B1} holds for $\eta>0$, and $\inf_{x \geq 0}\sigma^2(x)/x>0$ then 
\begin{itemize}
  \item If $\E\left[(\Theta -1)\ln \Theta \right] < \eta $, then for any $x>0$
 $$ \P_x(X_t>0)=\mathcal{O}\left(e^{\underline{r}(\E[\ln 1/\Theta]-\eta-1/2)t}\right), \quad (t \to \infty).$$
  \item If $\E\left[(\Theta -1)\ln \Theta \right] = \eta $, then for any $x>0$
 $$ \P_x(X_t>0)=\mathcal{O}\left(t^{-1/2} e^{\underline{r}(\E[\ln 1/\Theta]-\eta-1/2)t}\right), \quad (t \to \infty).$$
  \item If $\E\left[(\Theta -1)\ln \Theta \right] > \eta $, then for any $x>0$
 $$ \P_x(X_t>0)=\mathcal{O}\left(t^{-3/2} e^{\underline{r}(\E[\ln 1/\Theta]-\eta+ \E\left[(\Theta^\tau -1)\right])t}\right), \quad (t \to \infty),$$
 where $\tau\in[0,1)$ is the unique value such that $ \E[\ln (1/\Theta)]-\eta+ \E\left[ \Theta^\tau \ln \Theta \right]=0 $.
 \end{itemize}
 \item[ii)] If Condition \ref{B1} holds for $\eta=0$, and $\inf_{x \geq 0}\sigma^2(x)/x>0$ then for any $x>0$
 $$ \P_x(X_t>0)=\mathcal{O}\left(t^{-1/2}\right), \quad (t \to \infty).$$
\item[iii)] 
 Under the assumptions of point iii) of Proposition \ref{3_behaviours}, there exists a stochastic process $(K_t, t\geq 0)$, larger than a L\'evy process with drift $\eta$, and a non-decreasing function $\rho$ such that $\rho(t) \geq \uline{r} t$ and 
\begin{equation} \label{ineqW} 
\lim_{t \to \infty} X_te^{-K_{\rho(t)}} = W 
\end{equation}
where $W$ is a finite non-negative random variable satisfying $\P(W>0)>0$.
\end{itemize}
\end{cor}

Absorption rates of CSBPs in random environment have been intensively studied during the 
last decade \cite{boeinghoff2011branching,BPS,palau2016asymptotic,li2018asymptotic,bansaye2019extinction}.
In these references, $g(x)=gx$ , $\sigma^2(x)= \sigma^2 x$, for some $\sigma \geq 0$, $p(x)=x$ and $r(x) \equiv r$ 
is independent of $X$, whereas these assumptions 
are relaxed in our case (notice however that we make moment assumptions on the jump measures). Corollary \ref{iiimoreprecise} 
thus provides bounds on the survival probability for a new class of processes.\\

Let us finally describe the long time behaviour of the process $X$ under Condition \ref{B4}.

\begin{prop} \label{3_behaviours_expl}
Suppose that Assumption \ref{ass_PSS} is satisfied.
\begin{itemize}
 \item[i)] If Condition \ref{B4} holds for $\eta>0$, then 
 $$ \lim_{t \to \infty} X_t = \infty,  \quad \text{almost surely}.$$
\item[ii)] If Condition \ref{B4} holds for $\eta=0$, then 
 $$ \limsup_{t \to \infty} X_t = \infty,  \quad \text{almost surely}.$$
\end{itemize}
\end{prop}

This result describes quantitatively how much the growth of the process has to overcome its fluctuations to drift to infinity.\\

The rest of the paper is dedicated to the proofs.

\section{Proofs} \label{section_proofs}

Using recent results on SDEs with jumps, we first prove that the class of processes we are interested in may be realized as unique pathwise solutions to SDEs.

\subsection{Proofs of Section \ref{section_model}}

\begin{proof}[Proof of Proposition \ref{prop_sol_SDE}]
 The proof is a direct application of Proposition 1 in \cite{palau2018branching}.
First according to their conditions (i) to (iv) on page 60, our parameters are admissible.
Second, we need to check that conditions (a), (b) and (c) in \cite{palau2018branching} are fulfilled.

In our case, condition (a) writes as follows: for any $n \in \N$, there exists $A_n<\infty$ such that for any $0 \leq x \leq n$,
$$ \int_0^\infty \int_0^1 \left| (\theta-1)x \mathbf{1}_{\{z \leq r(x)\}} \right| \kappa(d\theta)dz
= xr(x) \int_0^1 (1-\theta)\kappa(d\theta)
\leq A_n(1+x). $$
The function $r$ is continuous, and thus bounded on $[0,n]$. As a consequence, condition (a) holds.

To verify condition (b), it is enough to check that for any $n \in \N$ there exists $B_n<\infty$ such that for 
$0 \leq x\leq y\leq n$,
$$ |g(x)-g(y)|+ \int_0^\infty \int_0^1 (1-\theta) \left| x \mathbf{1}_{\{u \leq r(x)\}} - y \mathbf{1}_{\{u \leq r(y)\}} \right| \kappa(d\theta) du
\leq B_n\phi(y-x). $$
Indeed, the function $r_n: z \mapsto B_n \phi(z) $ on $\R_+$ is concave and non-decreasing and satisfies $ \int_{0^+}r_n^{-1}(z)dz = \infty$.
Now we have the following:
\begin{align*}
&\int_0^\infty \left| x \mathbf{1}_{\{u \leq r(x)\}} - y \mathbf{1}_{\{u \leq r(y)\}} \right|du \\
& \quad =  \int_0^\infty \left( (y-x) \mathbf{1}_{\{u \leq (r(x) \wedge r(y))\}} + y \mathbf{1}_{\{r(x)<u \leq r(y)\}} + x \mathbf{1}_{\{r(y)<u \leq r(x)\}}  \right)du\\
& \quad =  \ \mathbf{1}_{\{ r(x)<r(y) \}} (yr(y)-xr(x))+\mathbf{1}_{\{ r(y)\leq r(x) \}} (yr(y)+xr(x)-2xr(y))\\
& \quad \leq |yr(y)-xr(x)|+r(y)(y-x)+x|r(x)-r(y)|.
\end{align*}
But recall that a function that is locally Lipschitz on a compact interval is Lipschitz on this interval. Hence, $r$ is Lipschitz on $[0,n]$, and
condition (b) holds under Assumption \ref{ass_PSS}.

Finally, let us focus on condition (c).
First, as $p$ is non-decreasing, the function $x\mapsto x + z\mathbf{1}_{\lbrace u\leq p(x)\rbrace}$ is non-decreasing for all $(z,u)\in\mathbb{R}_+^2$. Second, the following inequality must be satisfied: for any $n \in \N$ there exists $D_n<\infty$ such that for 
$0 \leq x,y\leq n$,
$$ |\sigma(x)-\sigma(y)|^2+  \int_{\mathbb{R}_+^2} 
\left( \left|\mathbf{1}_{\{u \leq p(x)\}}z -\mathbf{1}_{\{u \leq p(y)\}}z \right|\wedge 
\left|\mathbf{1}_{\{u \leq p(x)\}}z -\mathbf{1}_{\lbrace u \leq p(y)\rbrace}z \right|^2\right)\pi(dz)du \leq D_n|x-y|. $$
The first term fulfills the condition as $\sigma$ is H\"older continuous with index $1/2$.
The second term is equal to 
$$  \int_{0}^\infty(z\wedge z^2)\pi(dz)
\int_{0}^\infty  \left|\mathbf{1}_{\{u \leq p(x)\}} -\mathbf{1}_{\{u \leq p(y)\}}\right| du  =
\left(\int_{0}^\infty(z\wedge z^2)\pi(dz)\right)|p(x) - p(y)|,  $$
and we conclude using again that $p$ is Lipschitz on $[0,n]$. Hence, condition (c) is satisfied.
We can thus conclude that Proposition 1 in \cite{palau2018branching}  applies, which in particular justifies that $X$ admits the infinitesimal generator given in \eqref{inf_gene}.
\end{proof}

\subsection{Proofs of Section \ref{section_abs_lineage}}\label{sec:proofsec3}

Let us first prove Remark \ref{rem_simpli}.

\begin{proof}[Proof of Remark \ref{rem_simpli}]
Under the assumptions of Remark \ref{rem_simpli}, the integral corresponding to the positive jumps is bounded in the neighborhood of $0$. To show that, we divide the integral into two parts. 
First
\begin{multline*} \limsup_{x\rightarrow 0^+} \left( p(x)x^{-2}\int_0^x z^2\left(\int_0^1 (1+zx^{-1}v)^{-1-a}(1-v)dv\right)\pi(dz) \right)
\\ \leq \limsup_{x\rightarrow 0^+} \left( p(x)x^{-1} \int_0^x \frac{z}{x} z\pi(dz) \right) \leq \left(\limsup_{x\rightarrow 0^+}p(x)x^{-1} \right)
\int_{\R_+}z \pi(dz)<\infty,\end{multline*}
where we used that $p$ is locally Lipschitz on $\R_+$ and $p(0)=0$ (Assumption \ref{ass_PSS}) which implies that $p$ is Lipschitz on 
$[0,1]$ and that $x\mapsto p(x)/x$ is bounded in the vicinity of $0$.

For the second part, first, note that for all $x>0$ and $z\in [x,\infty)$
\begin{align*}
\int_0^1 (1+zx^{-1}v)^{-1-a}(1-v)dv & \leq \int_0^{x/z} dv+\int_{x/z}^1 (1+zx^{-1}v)^{-1-a}dv\\
& = \frac{x}{z}+ \frac{x}{az} \left[ 2^{-a} -(1+zx^{-1})^{-a} \right]\leq \frac{x}{z}\left(1+\frac{2^{-a}}{a}\right).
\end{align*}
Then, 
\begin{multline*}
\limsup_{x\rightarrow 0^+} \left( p(x)x^{-2}\int_x^\infty z^2\left(\int_0^1 (1+zx^{-1}v)^{-1-a}(1-v)dv\right)\pi(dz) \right) \\
\leq \left(\limsup_{x\rightarrow 0^+}p(x)x^{-1} \right) \int_0^\infty z\left[1+ \frac{1}{a}  2^{-a} \right] \pi(dz)<\infty.
\end{multline*}
Therefore, if $\int_{\mathbb{R}_+}z\pi(dz)<\infty$, the part in $G_a$ corresponding to the positive jumps does not affect the boundedness of 
$G_a$ in the vicinity of $0$.
\end{proof}

We now prove Theorem \ref{thm:absorption}. As mentioned previously, the proof 
uses ideas of the proof of \cite[Theorem 2.3]{li2017general}.
However, as we extend this theorem, several steps of the proof have to be modified.
For the sake of readability we provide the whole proof, including parts which were done similarly in \cite{li2017general}. 
The proof relies on a martingale, whose construction is detailed in the next lemma. Recall the definitions of $\tau^\pm$ in Equations \eqref{tau1} and \eqref{tau2}.

\begin{lemma}\label{prop:martingale}
Suppose that Assumption \ref{ass_PSS} holds. For all $b>c>0$, let $T = \tau^-(c)\wedge \tau^+(b)$. Then, for all $a\in\mathcal{A}\cup (0,1)$, the process 
$$
Z^{(a)}_{t\wedge T}:=\left(X_{t\wedge T}\right)^{1-a}\exp\left(\int_0^{t\wedge T} G_a\left(X_s\right)ds\right)
$$ 
is a $\mathcal{F}_t$-martingale.
\end{lemma}
 
 \begin{proof}[Proof of Lemma \ref{prop:martingale}]
We follow the ideas of the proof of \cite[Lemma 5.1]{li2017general}. Let $a\in\mathcal{A} \cup (0,1)$. 
Applying It\^o's formula with jumps (see for instance \cite[Theorem 5.1]{ikeda1989}), we have for all $t\geq 0$
\begin{align*}
X_t^{1-a}  = & X_0^{1-a}+\int_0^t\left[(1-a)\frac{g(X_s)}{X_s}X_s^{1-a}-(1-a)aX_s^{-a-1}\sigma^2\left(X_s\right)\right]ds\\
& +\int_0^t \int_{\mathbb{R}_+} p(X_{s})((z+X_{s})^{1-a}-X_{s}^{1-a}-(1-a)zX_{s}^{-a})\pi(dz)ds\\
& + \int_0^t\int_0^{r(X_{s^-})}\int_0^1 \left(\theta^{1-a}-1\right)X_{s^-}^{1-a}N(ds,dx,d\theta)\\
& + (1-a) \int_0^t X_s^{1-a} \sqrt{2\sigma^2(X_s)}dB_s + \int_0^t \int_0^{p(X_{s^-})}\int_{\R_+} 
\left[ (X_{s^-}+z)^{1-a}-X_{s^-}^{1-a} \right]\widetilde{Q}(ds,dx,dz)\\
& =  X_0^{1-a}-\int_0^t X_s^{1-a}G_a(X_s)ds+ M_t,
\end{align*}
where $G_a$ has been defined in \eqref{eq:Ga} and $\left(M_t, t\geq 0\right)$ is a local martingale. For the last equality, we used formula \eqref{eq:Ia2} for $I_a$.
Next, using integration by parts we get
\begin{align*}
Z^{(a)}_{t\wedge T}  = & X_0^{1-a}+\int_0^{t\wedge T}G_a\left(X_s\right)Z_sds+\int_0^{t\wedge T}\exp\left(\int_0^{s} G_a\left(X_r\right)dr\right)dM_s - \int_0^{t\wedge T}G_a\left(X_s\right)Z_sds,
\end{align*}
so that $\left(Z^{(a)}_{t\wedge T}, t\geq 0 \right)$ is a local martingale. Similarly to \cite{li2017general}, we have
$$
\mathbb{E}_{\eps}\left[\sup_{s\leq t} \left(X_{s\wedge T}\right)^{1-a}\exp\left(\int_0^{s\wedge T} G_a\left(X_r\right)dr\right)\right]<\infty,
$$
using Assumptions \ref{ass_PSS}, so that from \cite[Theorem 51 p.38]{protter2005stochastic}, $\left(Z^{(a)}_{t\wedge T}, t\geq 0 \right)$ is a martingale.
\end{proof}

\begin{proof}[Proof of Theorem \ref{thm:absorption}]

We first focus on point {\it i)}.
Let $n \in \N$ be such that $n\geq 2$ and let $0<\varepsilon<b<1$ and $a\in\mathcal{A}$ be such that  \eqref{condi_noext} holds for all $u\leq b$. Let $T_n = \tau^-(\varepsilon^n)\wedge \tau^+(b)$. 
According to Lemma \ref{prop:martingale}, $Z_{t\wedge T_n}^{(a)}$ is an $\mathcal{F}_t$-martingale.
 As in \cite{li2017general},	 using Fatou's lemma, we have
\begin{align}\label{eq:first_esti_martingale}
 \mathbb{E}_{\varepsilon}\left[X_{T_n}^{1-a}\exp\left(\int_0^{T_n} G_a\left(X_s\right)ds\right)\right]\leq \lim_{t\rightarrow +\infty} \mathbb{E}_{\varepsilon}\left[X_{t\wedge  T_n}^{1-a}\exp\left(\int_0^{t\wedge T_n} G_a\left(X_s\right)ds\right)\right]=\varepsilon^{1-a}.
\end{align}
Next
\begin{align}\label{eq:epsi_1}
\mathbb{E}_{\varepsilon}\left[X_{T_n}^{1-a}\exp\left(\int_0^{T_n} G_a\left(X_s\right)ds\right)\right]\geq \mathbb{E}_{\varepsilon}\left[X_{T_n}^{1-a}\exp\left( -T_n\left|\inf_{x\in [\eps^n, b]}{G}_a(x)\right|\right)\mathbf{1}_{\lbrace\tau^-(\varepsilon^n)<\tau^+(b)\rbrace}\right].
\end{align}
We distinguish three cases.
\begin{enumerate}
 \item If 
 $$ 0 \leq \inf_{x\in (0,b]}G_a(x)<\infty,  $$
 then
 $$
\left|\frac{\inf_{x\in [\varepsilon^n,b]}G_a(x)}{{\ln(\eps^n)}}\right|\leq \frac{G_a(b)}{n |\ln \eps|} . 
$$
 \item If 
 $$ -\infty < \inf_{x\in (0,b]}G_a(x)< 0,  $$
 then
 $$
\left|\frac{\inf_{x\in [\varepsilon^n,b]}G_a(x)}{{\ln(\eps^n)}}\right|\leq \frac{\left|\inf_{x\in (0,b]}G_a(x)\right|}{n |\ln \eps|} . 
$$
 \item If 
 $$  \inf_{x\in (0,b]}G_a(x)=-\infty,  $$
 then there exists a sequence $(\alpha_n, n\in \N)$ converging to $0$ as $n$ goes to $\infty$ and such that $\eps^n \leq \alpha_n \leq b$, and
 $$
\left|\frac{\inf_{x\in [\varepsilon^n,b]}G_a(x)}{{\ln(\eps^n)}}\right|= \frac{\left|G_a(\alpha_n)\right|}{n |\ln \eps|}
\leq \left| \frac{G_a(\alpha_n)}{\ln \alpha_n}\right|. 
$$
\end{enumerate}
In the three cases, we obtain according to \ref{A1}, 
$$
\left|\frac{\inf_{x\in [\varepsilon^n,b]}G_a(x)}{{\ln(\eps^n)}}\right|\xrightarrow[n\rightarrow +\infty]{} 0. 
$$

Let 
 $$
d_n:=\left|\frac{\ln\left(\eps^{(a-1)n/2}\right)}{\inf_{x\in [\varepsilon^n,b]}G_a(x)}\right| = 
\frac{a-1}{2}\left|\frac{\ln\varepsilon^{n}}{\inf_{x\in [\varepsilon^n,b]}G_a(x)}\right| \xrightarrow[n\rightarrow +\infty]{} +\infty.
$$
As $a>1$, we have $X_{T_n}^{1-a}\mathbf{1}_{\lbrace\tau^-(\varepsilon^n)<\tau^+(b)\rbrace}\geq (\varepsilon^n)^{1-a}\mathbf{1}_{\lbrace\tau^-(\varepsilon^n)<\tau^+(b)\rbrace}$ . Then, we get from \eqref{eq:first_esti_martingale} and \eqref{eq:epsi_1},
\begin{align*}
 \varepsilon^{1-a} & \geq \left(\varepsilon^n\right)^{1-a}\mathbb{E}_{\varepsilon}\left[\exp\left(-d_n \left|\inf_{x\in [\varepsilon^n,b]}G_a(x)\right|\right)\mathbf{1}_{\lbrace\tau^-(\varepsilon^n)<\tau^+(b)\wedge d_n\rbrace} \right]\\
& = \left(\varepsilon^n\right)^{1-a}\E\left[\exp\left(\ln\left(\eps^{(a-1)n/2}\right)\right)\right]\mathbb{P}_{\varepsilon}\left(\tau^-(\varepsilon^n)<\tau^+(b)\wedge d_n\right).
\end{align*}
We thus obtain
\begin{align*}
\mathbb{P}_{\varepsilon}\left(\tau^-(\varepsilon^n)<\tau^+(b)\wedge d_n\right)\leq \varepsilon^{(a-1)(n/2-1)}.
\end{align*}
By the Borel-Cantelli Lemma, we have
\begin{align}\label{eq:res_borel_cantelli}
\mathbb{P}_{\varepsilon}\left(\tau^-(\varepsilon^n)<\tau^+(b)\wedge d_n \quad \text{i.o.}\right)=0, 
\end{align}
where i.o. stands for infinitely often.
As a consequence we get that, $\mathbb{P}_{\varepsilon}$-a.s., 
$$
\tau^-(\varepsilon^n)\geq\tau^+(b)\wedge d_n
$$
for $n$ large enough. If there are infinitely many $n$ so that
\begin{align}\label{eq:maj_taun}
\tau^-(\varepsilon^n)\geq d_n,
\end{align}
then we have $\tau^-(0)=\infty$. If \eqref{eq:maj_taun} holds for at most finitely many $n$, 
then by \eqref{eq:res_borel_cantelli}, we have $\tau^-(\varepsilon^n)> \tau^+(b)$ for all $n$ large enough. 
We conclude that for all $0<\varepsilon<b$,
\begin{align}\label{eq:dichotomie}
\mathbb{P}_{\varepsilon}\left(\tau^-(0) =\infty\text{ or }\tau^+(b)<\tau^-(0)\right)=1.
\end{align}

We will now use a coupling to show that $\P_\eps(\tau^-(0)<\infty)=0$. 
Let for $N \in \N$,
$$ r_{[0,N]}:= \sup_{0 \leq x \leq N} r(x), $$
which is finite as $r$ is a continuous function.
Let
$\tX$ be the unique strong solution to
\begin{align*}\label{eq:coupling}\nonumber
\tX_t  = & \tX_0 + \int_0^t g(\tX_s) ds+ \int_0^t \sqrt{2 \sigma^2(\tX_s) }dB_s 
+ \int_0^t \int_0^{p(\tX_{s^-})}\int_{\mathbb{R}_+}z\widetilde{Q}(ds,dx,dz)
\\ &+ \int_0^t \int_0^{r_{[0,N]}} \int_0^1  (\theta-1)\tX_{s^-}N(ds,dx,d\theta),
\end{align*}
where the Brownian motion $B$ and the Poisson random measures $\widetilde{Q}$ and $N$ are the same as in \eqref{eq:EDS}. 
We will use four properties of this equation. 
\begin{itemize}
 \item[a)] It has a unique strong solution according to Proposition \ref{prop_sol_SDE}.
 \item[b)] If $\tX^{(1)}$ and $\tX^{(2)}$ are two solutions with $\tX^{(1)}_0 \leq \tX^{(2)}_0$, then
 $\tX^{(1)}_t \leq \tX^{(2)}_t$ for any positive $t$.
 \item[c)] If $\tX$ is a solution with $\tX_0= X_0$, then $\tX_t\leq X_t$ for any $t$ smaller than $\tau^-(0) \wedge \tau^+(N) $.
 \item[d)] Equation \eqref{eq:dichotomie} holds for both $X$ and $\tX$.
\end{itemize}

Our aim now is to prove that 
\begin{equation} \label{concltX}
 \P_\eps\left( \ttau^-(0)<\infty \right)=0,
\end{equation}
where the $\ttau$'s are defined as the $\tau$'s in \eqref{tau1} and \eqref{tau2} but for the process $\widetilde{X}$.
Using the coupling described in point c), it will imply that
$$  \P_\eps\left( \tau^+(N) \leq \tau^-(0) \right)=1,
$$
and letting $N$ tend to infinity, we will get
$$  \P_\eps\left( \tau^-(0) =\infty \right)=1.
 $$

Before proceeding to the proof of \eqref{concltX}, let us notice that from coupling b) we have:
\begin{equation} \label{bleqbeta}
 \mathbb{E}_{\mathfrak{b}}\left[e^{-\lambda\ttau^-(\varepsilon)}\mathbf{1}_{\lbrace \ttau^-(\varepsilon)<\infty\rbrace}\right] 
 \leq \mathbb{E}_{b}\left[e^{-\lambda\ttau^-(\varepsilon)}\mathbf{1}_{\lbrace \ttau^-(\varepsilon)<\infty\rbrace}\right] 
 \quad \forall \ b \leq \mathfrak{b}.
\end{equation}
Now the strategy to prove \eqref{concltX} will be to show that for any $\lambda > 0 $
$$ \mathfrak{A}(\lambda,\eps) := \int_0^1 \mathbb{E}_{\theta \eps}\left[e^{-\lambda\ttau^-(0)}\mathbf{1}_{\lbrace \ttau^-(0)< \infty\rbrace}\right] \kappa(d\theta) = 0.$$

For any $0<\theta \leq 1$, \eqref{eq:dichotomie} yields
\begin{align*}
\mathbb{E}_{\theta \eps}\left[e^{-\lambda\ttau^-(0)}\mathbf{1}_{\lbrace \ttau^-(0)< \infty\rbrace}\right]&  = 
\mathbb{E}_{\theta \eps}\left[e^{-\lambda\ttau^-(0)}\mathbf{1}_{\lbrace \ttau^+(b)<\ttau^-(0)<\infty\rbrace}\right]\\
&\leq \mathbb{E}_{\theta \varepsilon}\left[e^{-\lambda\ttau^+(b)}\mathbf{1}_{\lbrace \ttau^+(b)<\ttau^-(0)\rbrace}\right]\mathbb{E}_{b}\left[e^{-\lambda\ttau^-(0)}\mathbf{1}_{\lbrace \ttau^-(0)<\infty\rbrace}\right],
\end{align*}
where the last inequality comes from the Markov property combined with \eqref{bleqbeta}. Moreover, using again the Markov property, we have
\begin{align*}
\mathbb{E}_{b}\left[e^{-\lambda\ttau^-(0)}\mathbf{1}_{\lbrace \ttau^-(0)<\infty\rbrace}\right] & =
\mathbb{E}_{b}\left[e^{-\lambda\ttau^-(\varepsilon)}\mathbf{1}_{\lbrace \ttau^-(\varepsilon)<\infty\rbrace}\mathbb{E}_{X_{\ttau^-(\varepsilon)}}\left[e^{-\lambda\ttau^-(0)}\mathbf{1}_{\lbrace \ttau^-(0)<\infty\rbrace}\right]\right].
\end{align*}
The process can cross the level $\varepsilon$ either because of the diffusion or because of a negative jump. 
In both cases, $X_{\ttau^-(\varepsilon)}\geq \eps \Theta $ almost surely, 
 where we recall that $\Theta$ is a random variable distributed according to $\kappa$ and independent of the process before time $\ttau^-(\varepsilon)$. Then, using again \eqref{bleqbeta},

\begin{align*}
\mathbb{E}_{b}\left[e^{-\lambda\ttau^-(0)}\mathbf{1}_{\lbrace \ttau^-(0)<\infty\rbrace}\right] & \leq
\mathbb{E}_{b}\left[e^{-\lambda\ttau^-(\varepsilon)}\mathbf{1}_{\lbrace \ttau^-(\varepsilon)<\infty\rbrace}\right]
 \int_0^1 \mathbb{E}_{\theta \eps}\left[e^{-\lambda\ttau^-(0)}\mathbf{1}_{\lbrace \ttau^-(0)< \infty\rbrace}\right] \kappa(d\theta) .
\end{align*}

We thus get
\begin{align*}
\mathfrak{A}(\lambda,\eps)\leq &  \mathbb{E}_{b}\left[e^{-\lambda\ttau^-(\varepsilon)}\mathbf{1}_{\lbrace \ttau^-(\varepsilon)<\infty\rbrace}\right]
\mathfrak{A}(\lambda,\eps).
\end{align*}

As 
$$  \mathbb{E}_{b}\left[e^{-\lambda\ttau^-(\varepsilon)}\mathbf{1}_{\lbrace \ttau^-(\varepsilon)<\infty\rbrace}\right]<1, $$
we conclude that $\mathfrak{A}(\lambda,\eps)= 0,$ which ends the proof of point {\it i)}.\\

Let us now focus on point {\it ii)}.
Let $\delta<(3-2a)^{-1}$ and $\eps>0$ such that 
\begin{equation} \label{epspetit} \varepsilon < e^{-(1-\delta)^{-1}}\wedge\left(\left(\frac{\ln 2}{2}\right)^{1/\delta(1-a)}\left( \frac{1}{\E[\Theta^{1-a}]^{-1} +1} \right)^{1/\delta(1-a)}\right).
\end{equation}
Let $T = \tau^-(\varepsilon^{1+\delta})\wedge \tau^+(\varepsilon^{1-\delta})$. Finally, let $0<a<1$ and $\eta>0$ be such that Condition \ref{A2} is satisfied. Then, as in the proof of point {\it i)}, we have for all $z>0$ such that $\varepsilon^{1+\delta}<z < \varepsilon^{1-\delta}$,
\begin{align*}
 z^{1-a} & \geq \mathbb{E}_{z}\left[X_{ \tau^+(\varepsilon^{1-\delta})}^{1-a}\exp\left(\int_0^{ \tau^+(\varepsilon^{1-\delta})} G_a\left(X_u\right)du\right)\mathbf{1}_{\lbrace \tau^+(\varepsilon^{1-\delta})< \tau^-(\varepsilon^{1+\delta})\rbrace}\right]\\
& \geq \varepsilon^{(1-\delta)(1-a)} \mathbb{P}_{z}\left(\tau^+(\varepsilon^{1-\delta})< \tau^-(\varepsilon^{1+\delta})\right),
\end{align*}
where we have used that if \eqref{condi_ext} holds, then $G_a(z)\geq 0$ for $z<\eps^{1-\delta}<e^{-1}$.
Therefore,
\begin{align}\label{eq:tempsarret1}
\mathbb{P}_{z}\left(\tau^+(\varepsilon^{1-\delta})< \tau^-(\varepsilon^{1+\delta})\right) \leq  \varepsilon^{(\delta-1)(1-a)}z^{1-a}. 
\end{align}
Similarly, for every $t\geq 0$
\begin{align*}
 z^{1-a} & \geq \mathbb{E}_{z}\left[X_t^{1-a}\exp\left(\int_0^{ t} G_a\left(X_s\right)ds\right)\mathbf{1}_{\lbrace \tau^+(\varepsilon^{1-\delta})= 
 \tau^-(\varepsilon^{1+\delta})=\infty\rbrace}\right]\\
 & \geq \eps^{(1+\delta)(1-a)} e^{t \ln (\eps^{-(1-\delta)})\ln(\ln(\eps^{-(1-\delta)}))^{1+\eta}} \mathbb{P}_{z}\left(\tau^+(\varepsilon^{1-\delta})= \tau^-(\varepsilon^{1+\delta})=\infty\right) ,
\end{align*}
so that
\begin{align*}
 \mathbb{P}_{z}\left(\tau^+(\varepsilon^{1-\delta})= \tau^-(\varepsilon^{1+\delta})=\infty\right) \leq z^{1-a} \varepsilon^{-(1+\delta)(1-a)}\exp\left(-t\ln(\varepsilon^{-(1-\delta)})\ln(\ln(\eps^{-(1-\delta)}))^{1+\eta}\right).
\end{align*}
Letting $t$ tend to infinity yields
\begin{align}\label{eq:tempsarret2}
 \mathbb{P}_{z}\left(\tau^+(\varepsilon^{1-\delta})= \tau^-(\varepsilon^{1+\delta})=\infty\right) = 0.
\end{align}

Let 
\begin{equation} t(\eps): = \left(\ln (\ln(\eps^{-(1-\delta)}))\right)^{-1-\eta} .
\label{teps}\end{equation}
 We have, using \eqref{condi_ext},
\begin{align*}
 z^{1-a}  \geq & \mathbb{E}_{z}\left[X_{ \tau^-(\varepsilon^{1+\delta})}^{1-a}\exp\left(\int_0^{ \tau^-(\varepsilon^{1+\delta})} G_a\left(X_u\right)du\right)\mathbf{1}_{\lbrace t(\varepsilon)<\tau^-(\varepsilon^{1+\delta})< \tau^+(\varepsilon^{1-\delta})\rbrace}\right]\\
 \geq & \exp\left[\ln(\eps^{-(1-\delta)})\left(\ln (\ln(\eps^{-(1-\delta)}))\right)^{1+\eta} t(\varepsilon)\right]
 \mathbb{E}_{z}\left[\left(X_{ \tau^-(\varepsilon^{1+\delta})}\right)^{1-a}
 \mathbf{1}_{\lbrace t(\varepsilon)<\tau^-(\varepsilon^{1+\delta})< \tau^+(\varepsilon^{1-\delta})\rbrace}\right]\\
= & \varepsilon^{-(1-\delta)}
 \mathbb{E}_{z}\left[X_{ \tau^-(\varepsilon^{1+\delta})}^{1-a}
 \mathbf{1}_{\lbrace t(\varepsilon)<\tau^-(\varepsilon^{1+\delta})< \tau^+(\varepsilon^{1-\delta})\rbrace}\right]\\
\geq & \varepsilon^{-(1-\delta)}
 \varepsilon^{(1+\delta)(1-a)} \E[\Theta^{1-a}] \mathbb{P}_{z}\left(
  t(\varepsilon)<\tau^-(\varepsilon^{1+\delta})< \tau^+(\varepsilon^{1-\delta})\right),
\end{align*}
where we used as before that for all $y\geq 0$, $X_{\tau^-(y)}\geq y\Theta$ almost surely where $\Theta$ is a random variable distributed according to $\kappa$ 
independent of the process before time $\tau^-(y)$.
We deduce,
\begin{align}\label{eq:tempsarret3}
\mathbb{P}_{z}\left( t(\varepsilon)<\tau^-(\varepsilon^{1+\delta})< \tau^+(\varepsilon^{1-\delta})\right) \leq \E[\Theta^{1-a}]^{-1}\eps^{a+(a-2)\delta}z^{1-a}.
\end{align}
Combining \eqref{eq:tempsarret1}, \eqref{eq:tempsarret2} and \eqref{eq:tempsarret3}, we get for all $z>0$ such that $\varepsilon^{1+\delta}<z < \varepsilon^{1-\delta}$,
\begin{align}\nonumber\label{eq:maj_proba}
\mathbb{P}_{z}\left(\tau^-(\varepsilon^{1+\delta})> t(\varepsilon)\right) 
& \leq \E[\Theta^{1-a}]^{-1}\eps^{a+(a-2)\delta}z^{1-a}+ \eps^{(\delta-1)(1-a)}z^{1-a}\\
& = \eps^{(\delta-1)(1-a)}z^{1-a}\left( \E[\Theta^{1-a}]^{-1}\eps^{1-\delta(3-2a)} +1\right) \nonumber \\
& \leq \left( \E[\Theta^{1-a}]^{-1} +1\right)\left(\eps^{(\delta-1)}z\right)^{1-a},
\end{align}
as by assumption $\delta$ is smaller than $(3-2a)^{-1}$.
By the strong Markov property,
\begin{align}\label{eq:prod}
& \mathbb{P}_{z}\left(\bigcap_{n=0}^{m}\left\lbrace \tau^-(\varepsilon^{(1+\delta)^n})<\infty,\tau^-(\varepsilon^{(1+\delta)^{n+1}})\circ\theta_{\tau^-(\varepsilon^{(1+\delta)^n})}\leq t(\varepsilon^{(1+\delta)^n})\right\rbrace\right)\\
& = \mathbb{E}_{z}\left[\prod_{n=0}^m \mathbb{P}_{X_{\tau^-\left(\eps^{(1+\delta)^{n}}\right)}}\left(\tau^-(\varepsilon^{(1+\delta)^{n+1}})\leq t(\varepsilon^{(1+\delta)^n})\right)\right],\nonumber
\end{align}
where $\theta_s:\mathbb{D}(\mathbb{R}_+,\mathbb{R}_+)\rightarrow\mathbb{D}(\mathbb{R}_+,\mathbb{R}_+)$ is the shift operator $(\theta_sX)(t) = X(s+t)$.

There are two possibilities. If 
 $$ X_{\tau^-(\varepsilon^{(1+\delta)^{n}})}>\eps^{(1+\delta)^{n+1}}, $$
then we can apply \eqref{eq:maj_proba} with $\eps^{(1+\delta)^n}$ instead of $\eps$, and we get
\begin{align*}
\mathbb{P}_{X_{\tau^-(\varepsilon^{(1+\delta)^{n}})}}\left(\tau^-(\varepsilon^{(1+\delta)^{n+1}})\leq t(\varepsilon^{(1+\delta)^n})\right) 
&\geq 1- \left( \E[\Theta^{1-a}]^{-1} +1\right)\left(\eps^{(\delta-1)(1+\delta)^n}X_{\tau^-(\varepsilon^{(1+\delta)^{n}})}\right)^{1-a}\\
&\geq 1- \left( \E[\Theta^{1-a}]^{-1} +1\right)\left(\eps^{(\delta-1)(1+\delta)^n}\varepsilon^{(1+\delta)^{n}}\right)^{1-a}\\
& = 1- \left( \E[\Theta^{1-a}]^{-1} +1\right)\left(\eps^{\delta(1+\delta)^n}\right)^{1-a}.
\end{align*}
Else if
 $$ X_{\tau^-(\varepsilon^{(1+\delta)^{n}})}\leq \eps^{(1+\delta)^{n+1}}, $$
then
\begin{align*}
\mathbb{P}_{X_{\tau^-(\varepsilon^{(1+\delta)^{n}})}}\left(\tau^-(\varepsilon^{(1+\delta)^{n+1}})\leq t(\varepsilon^{(1+\delta)^n})\right)  = 1\geq 1- \left( \E[\Theta^{1-a}]^{-1} +1\right)\left(\eps^{\delta(1+\delta)^n}\right)^{1-a}.
\end{align*}
Combining this inequality with \eqref{eq:prod}, we thus obtain
\begin{align}\nonumber\label{eq:1}
& \mathbb{P}_{z}\left(\bigcap_{n=0}^{m}\lbrace \tau^-(\varepsilon^{(1+\delta)^n})<\infty,\tau^-(\varepsilon^{(1+\delta)^{n+1}})\circ\theta_{\tau^-(\varepsilon^{(1+\delta)^n)})}\leq t(\varepsilon^{(1+\delta)^n})\rbrace\right)\\\nonumber
& \geq \prod_{n=0}^m\left(1- \left( \E[\Theta^{1-a}]^{-1} +1\right)\varepsilon^{(1-a)\delta(1+\delta)^n}\right)\\
& \geq \prod_{n=0}^me^{-2\left( \E[\Theta^{1-a}]^{-1} +1\right)\varepsilon^{(1-a)\delta(1+\delta)^n}}
 = e^{-2\left( \E[\Theta^{1-a}]^{-1} +1\right)\sum_{n=0}^m\varepsilon^{(1-a)\delta(1+\delta)^n}}
\end{align}
where for the last inequality, we used 
\begin{equation*} \varepsilon < \left(\frac{\ln 2}{2}\right)^{1/\delta(1-a)}\left( \frac{1}{\E[\Theta^{1-a}]^{-1} +1} \right)^{1/\delta(1-a)},
\end{equation*}
and that $x\mapsto 1-x-e^{-2x}$ is positive for $0<x\leq (\ln 2)/2$. Next,
\begin{align}\label{eq:2}
\sum_{n=0}^m\varepsilon^{(1-a)\delta(1+\delta)^n}=\varepsilon^{(1-a)\delta}\sum_{n=0}^m\varepsilon^{(1-a)\delta((1+\delta)^n-1)}\leq \varepsilon^{(1-a)\delta}\sum_{n=0}^m\varepsilon^{(1-a)\delta^2n}\leq \frac{\varepsilon^{(1-a)\delta}}{1-\varepsilon^{(1-a)\delta^2}}.
\end{align}
Combining \eqref{eq:1} and \eqref{eq:2} and letting $m\rightarrow \infty$, we get by monotone convergence
\begin{align*}
& \mathbb{P}_{z}\left(\bigcap_{n=0}^{\infty}\left\lbrace \tau^-(\varepsilon^{(1+\delta)^n})<\infty,\tau^-(\varepsilon^{(1+\delta)^{n+1}})\circ\theta_{\tau^-(\varepsilon^{(1+\delta)^n)})}\leq t(\varepsilon^{(1+\delta)^n})\right\rbrace\right)\\
&\geq e^{-2\left( \E[\Theta^{1-a}]^{-1} +1\right)\varepsilon^{(1-a)\delta}(1-\varepsilon^{(1-a)\delta^2})^{-1}}.
\end{align*}
Since under $\mathbb{P}_z$, 
\begin{align*}
\tau^-(0) = \sum_{n=0}^\infty \tau^-(\varepsilon^{(1+\delta)^{n+1}})\circ\theta_{\tau^-(\varepsilon^{(1+\delta)^n})},
\end{align*}
then
\label{pageineq}
\begin{align*}
\mathbb{P}_z\left(\tau^-(0)\leq \sum_{n=0}^\infty t(\varepsilon^{(1+\delta)^n})\right)\geq e^{-2\left( \E[\Theta^{1-a}]^{-1} +1\right)\varepsilon^{(1-a)\delta}(1-\varepsilon^{(1-a)\delta^2})^{-1}}.
\end{align*}
Notice that for $\varepsilon_n:= \varepsilon^{(1+\delta)^n}$,
$$ t(\eps_n)= \left( \ln(\ln (\eps^{-(1-\delta)(1+\delta)^n})) \right)^{-(1+\eta)}
= \left( n \ln (1+\delta)+ \ln (1-\delta) + \ln(\ln (\eps^{-1}) \right)^{-(1+\eta)}.$$
In particular, for large $n$, 
$$ t(\eps_n) \sim \left( n \ln (1+\delta) \right)^{-(1+\eta)}. $$
This ensures that
\begin{align*}
\sum_{n=1}^\infty t(\varepsilon_n) <\infty.
\end{align*}
We thus have
\begin{align*} 
\mathbb{P}_z\left(\tau^-(0)-<\infty\right)\geq e^{-2\left( \E[\Theta^{1-a}]^{-1} +1\right)\varepsilon^{(1-a)\delta}(1-\varepsilon^{(1-a)\delta^2})^{-1}}>0.
\end{align*}
This ends the proof of point {\it ii)}.

We now prove point {\it iii).}
Assume that for any positive $x$, $r(x)>0$. Let $x_0> 0$ be such that 
$$ \P_{y}(\tau^-(0)<\infty)>0, \quad \forall y \leq x_0. $$
Let $y> x_0$. 
By continuity, we can define:
$$ \sup_{x \leq 2y}p(x)=:\overline{p}<\infty, \quad \sup_{x \leq 2y}\sigma(x)=:\overline{\sigma}<\infty, $$
$$ \sup_{x \leq 2y}g(x)=:\overline{g}<\infty \quad \text{and} \quad \inf_{x \leq 2y}r(x)=:\underline{r}>0. $$
Moreover, there exists $\nu <1$ such that $\kappa([0,\nu])=:\nu_\kappa >0$, and we can take $N_{\kappa} \in \N$ such that 
\begin{equation} \label{descente} 2y\nu_\kappa^{N_\kappa}\leq x_0 .\end{equation}
Notice that by the Dubins–Schwarz Theorem \cite[Theorem 42 p.88]{protter2005stochastic} and the reflection principle, for any $t>0$ and $A\in \mathbb{R}$, 
\begin{align*}
\P_y \left( \sup_{s \leq t \wedge \tau^+(2y)} \int_0^s \sqrt{2\sigma^2(X_u)}dB_u >A \right) &=
2\P_y \left(W_{\int_0^{t \wedge \tau^+(2y)} 2\sigma^2(X_u)du}>A \right)\\
&= 2 \P_y \left(W_1>\frac{A}{\sqrt{{\int_0^{t \wedge \tau^+(2y)} 2\sigma^2(X_u)du}}} \right)\\
&\leq 2 \P_y \left(W_1>\frac{A}{\sqrt{2t\overline{\sigma}^2}} \right).
\end{align*}
where $W$ is a standard Brownian motion.
Finally, let $J(t,y)$ denote the event of having no positive jumps during the time interval $[0,t \wedge \tau^+(2y)]$. Its probability is larger than $e^{-t\overline{p}}$. We denote by $\overline{J(t,y)}$ the complementary event.
Then, we have for all $t,y>0$,
\begin{align*}
\P_y(\tau^+(2y)\leq t)& = \P_y\left(X_{ t \wedge \tau^+(2y)}>2y,\ \overline{J(t,y)}\right) + \P_y\left(X_{ t \wedge \tau^+(2y)}>2y,\ J(t,y)\right)\\
&\leq \P_y\left(\overline{J(t,y)}\right) + \P_y\left(\overline{g}t+\sup_{0\leq s\leq t}\int_0^s \sqrt{2\sigma^2(X_u)}dBu>y\right)\\
&\leq 1-e^{-\overline{p}t}+2\P_y\left(W_1>\frac{y-\overline{g}t}{\sqrt{2t\overline{\sigma}^2}}\right):=1-e^{-\overline{p}t} + 2A(t,y),
\end{align*}
and for $y>x_0$,
\begin{align*}
A(t,y) = \P_0\left(W_1 +y>\frac{y-\overline{g}t}{\sqrt{2t\overline{\sigma}^2}}\right) = \P_0\left(W_1>\frac{(1-\sqrt{2t\overline{\sigma}^2})y-\overline{g}t}{\sqrt{2t\overline{\sigma}^2}}\right)\leq A(t,x_0),
\end{align*}
where the final bound holds for $t\leq (2\overline{\sigma}^2)^{-1}$.
Finally, we obtain that there exists $t_{x_0}>0$
such that for all $y>x_0$
$$\P_y(\tau^+(2y)\leq t_{x_0})<1/2.$$
Moreover, the probability that during the time $t_{x_0}$, conditionally on $\lbrace \tau^+(2y)> t_{x_0}\rbrace$, the process makes at least $N_\kappa$ negative jumps, with a jump size in $(0,\nu]$ is larger than:
$$ p_{t_{x_0}}(\underline{r},N_\kappa,\nu_\kappa):=e^{- \underline{r}\nu_\kappa t_{x_0}}\sum_{i=N_\kappa}^\infty \frac{(\underline{r}\nu_\kappa t_{x_0})^i}{i!}. $$
This entails, using \eqref{descente},
\begin{equation} \label{temps_descente} \P_y(\tau^-(x_0)\leq t_{x_0})>p_{t_{x_0}}(\underline{r},N_\kappa,\nu_\kappa)/2>0, \end{equation}
which ends the proof of {\it iii)}.
\end{proof}

\subsection{Proofs of Section \ref{section_expl_lineage}}

We now focus on the explosion behaviour of the process.

\begin{proof}[Proof of Theorem \ref{thm:explosion}]
 As we gave all the details of the proof of Theorem \ref{thm:absorption}, we will here only provide the elements of the proof which differ from the proof of \cite[Theorem 2.8]{li2017general}.
 
We take a small enough $b^{-1}$ and
$\eps$ satisfying $0 < b < \eps^{-1}$.
 We begin the proof of point $i)$ similarly as in \cite[Theorem 2.8]{li2017general}, except that we take 
 $$ d_n:= \left|\frac{\ln\left(\eps^{(1-a)n/2}\right)}{\inf_{x\in [b,\varepsilon^{-n}]}G_a(x)}\right|$$
 instead of
 $$ d_n:= \frac{\ln (\eps^{-n(1-a)/2})}{\ln^r (\eps^{-n})}, $$
 and obtain in the same way using the Borel Cantelli lemma that
\begin{equation} \label{BorelCant2}
\P_{\eps^{-1}}(\tau^+(\infty)=\infty \text{ or } \tau^-(b) < \tau^+(\infty)<\infty )=1. \end{equation}
The authors of \cite{li2017general} then claim that they can conclude the proof as the proof of point $i)$ of their Theorem 2.3. However, in the latter case, they only need the strong Markov property to obtain that for any $\lambda >0$, $\E_\eps[e^{-\lambda \tau^-(0)} ;\tau^-(0)< \infty] = 0$ and
consequently, $\P_\eps (\tau^-(0)<\infty) = 0$, as their process does not have negative jumps.
In the current case, to obtain that $\E_{\eps^{-1}} [ e^{-\lambda \tau^+(\infty)}; \tau^+(\infty)<\infty ]=0$, we (and they) have to take into account the fact that there are positive jumps and that $X_{\tau^+(\eps^{-1})}$ may be strictly bigger than $\eps^{-1}$.

Let us first notice that for any $\eps^{-1} \leq y \leq 2 \eps^{-1}$, the same reasoning
as the one to obtain \eqref{BorelCant2}
leads to
\begin{equation}
 \label{two_cases}  
\P_{y}(\tau^+(\infty)=\infty \text{ or } \tau^-(b) < \tau^+(\infty)<\infty )=1. \end{equation}

Let us thus fix $\lambda>0$ and introduce the following real number:
$$ \mathcal{A}(\eps) := \sup_{\eps^{-1} \leq y \leq 2 \eps^{-1}}
 \E_y \left[ e^{-\lambda \tau^+(\infty)}; \tau^+(\infty)<\infty \right] . $$
For any $\eps < 1$, $y \leq \eps^{-1}$, we have by the Markov inequality
\begin{align*} \P_y(X_{\tau^+(\eps^{-1})}>2\eps^{-1}) & \leq 
\P_y(X_{\tau^+(\eps^{-1})}-X_{\tau^+(\eps^{-1})-}>\eps^{-1})\\
& = 
\P_y\left(\left(X_{\tau^+(\eps^{-1})}-X_{\tau^+(\eps^{-1})-}\right)^2\wedge \left(X_{\tau^+(\eps^{-1})}-X_{\tau^+(\eps^{-1})-}\right)>\eps^{-1} \right)\\& \leq 
\eps \int_0^\infty (z\wedge z^2)\pi(dz).
\end{align*}

Using Equation \eqref{two_cases} and the strong Markov property, we thus get, for any $\eps^{-1} \leq y \leq 2 \eps^{-1}$:
\begin{align*}
 \E_y \left[ e^{-\lambda \tau^+(\infty)}; \tau^+(\infty)<\infty \right]&=
 \E_y \left[ e^{-\lambda \tau^+(\infty)}; \tau^-(b) <\tau^+(\infty)<\infty \right]\\
 & \leq \E_y \left[ e^{-\lambda \tau^-(b)} \E_{X_{\tau^-(b)}}\left[e^{-\lambda \tau^+(\infty)}; \tau^+(\infty)<\infty\right]; \tau^-(b) <\infty\right].
\end{align*}
Using again the strong Markov property, we get for all 
$x\leq b\leq \eps^{-1}$
\begin{align*}
\E_{x}\left[e^{-\lambda \tau^+(\infty)}; \tau^+(\infty)<\infty\right] & = \E_y \left[ e^{-\lambda \tau^+(\infty)}; \tau^-(b) <\tau^+(\infty)<\infty \right]\\
&=\E_y \left[ e^{-\lambda \tau^+(\infty)}; \tau^-(b) <\tau^+(\eps^{-1})<\tau^+(\infty)<\infty \right]\\
&= \E_{x} \left[ e^{-\lambda \tau^+(\eps^{-1})} \E_{X_{\tau^+(\eps^{-1})}}\left[e^{-\lambda \tau^+(\infty)}; \tau^+(\infty)<\infty\right]; \tau^+(\eps^{-1}) <\infty\right]\\
& \leq \mathcal{A}(\eps)+\eps \int_0^\infty (z\wedge z^2)\pi(dz),
\end{align*}
where the last inequality is obtained by considering the event $\lbrace \eps^{-1}\leq X_{\tau^+(\eps^{-1})}\leq 2\eps^{-1}\rbrace$ and its complement.
Finally, combining the last two inequalities, we obtain
\begin{align*}
 \E_y \left[ e^{-\lambda \tau^+(\infty)}; \tau^+(\infty)<\infty \right] & \leq \E_y \left[ e^{-\lambda \tau^-(b)}; \tau^-(b) <\infty \right] \left(\mathcal{A}(\eps)+\eps \int_0^\infty (z\wedge z^2)\pi(dz)\right).
\end{align*}
But there exists $C(b)<1$ such that for $2b \leq y$, 
$$ \E_y \left[ e^{-\lambda \tau^-(b)}; \tau^-(b) <\infty \right]<C(b). $$
Otherwise we would have
$$ \lim_{y \to \infty} \E_y \left[ e^{-\lambda \tau^-(b)}; \tau^-(b) <\infty \right] =1, $$
and thus $\tau^-(b)$ would converge to $0$ when the initial condition of the process goes to $\infty$ which would contradict our assumptions on the regularity
of the negative jumps. Hence, as for $\eps$ small enough, $2b \leq \eps^{-1}$, we obtain for such an $\eps$
$$ \mathcal{A}(\eps) \leq \frac{C(b)\left( \int_{\R_+}(z \wedge z^2)\pi(dz)\right)\eps}{1-C(b)}. $$
We thus deduce that
$$ \lim_{y \to \infty}\E_y \left[ e^{-\lambda \tau^+(\infty)}; \tau^+(\infty)<\infty \right]=0. $$
Now, let us take $x,\mu>0$. Then, there exists $N_0$ such that for any $N \geq N_0$,
$$ \E_N \left[ e^{-\lambda \tau^+(\infty)}; \tau^+(\infty)<\infty \right]\leq \mu. $$
Hence,
$$ \E_x \left[ e^{-\lambda \tau^+(\infty)}; \tau^+(\infty)<\infty \right]
\leq \E_x \left[ \E_{X_{\tau^+(N_0)}} \left[ e^{-\lambda \tau^+(\infty)}; \tau^+(\infty)<\infty\right] \right]\leq \mu $$
and thus for all $x>0$
$$ \E_x \left[ e^{-\lambda \tau^+(\infty)}; \tau^+(\infty)<\infty \right]=0,$$
which completes the proof of point $i)$ (and of point $i)$ of \cite[Theorem 2.8]{li2017general}).\\

The proof of point $ii)$ is the same as the proof of point $ii)$ of \cite[Theorem 2.8]{li2017general}, except that we modify the function $t(\cdot)$ as we did for the proof of point $ii)$ of Theorem \ref{thm:absorption}.

We now prove point {\it iii).}
Assume 
 that for any positive $x$, $p(x)+ \sigma(x)>0$. Let $x_1> 0$ be such that 
$$ \P_{y}(\tau^+(\infty)<\infty)>0, \quad \forall y \geq x_1. $$
Let $y< x_1$. If $p(y)>0$, there exists $\eta_1>0$ such that $p$ stays positive on $[y-\eta_1,y+\eta_1]$ as it is a continuous 
function according to Assumption \ref{ass_PSS}. 
Hence, for $\eta_2>0$ small enough, starting from $y$, we can show as in the proof of Theorem \ref{thm:absorption}{\it iii)}, that the probability that the process is bigger than $x_1$ thanks to a positive jump is positive:
$$ \P_y(X_{\eta_2}\geq x_1)>0, $$
and using the Markov property, we obtain 
\begin{equation*} \P_y(\tau^+(\infty)<\infty)>0.
\end{equation*}
Now assume that $p(y)=0$ but $\sigma(y)>0$. As $\sigma$ is continuous, if $\sigma(z)>0$ for $z\in [y,x_1]$ then
$$\mathbb{P}_y(X_s\geq x_1)>0,\quad \forall s>0
$$
thanks to the diffusion and we end the proof by applying again the Markov property. Else, if $\sigma$ is only positive on an interval of the form 
$[y,x_2)$ with $y<x_2<x_1$, then by continuity of $p$ and $\sigma$ given by Assumption \ref{ass_PSS}, $p(x_2)>0$ and we are back to 
the first case.
We thus have proven that
\begin{equation*} \P_y(\tau^+(\infty)<\infty)>0,\quad \forall y\geq 0,
\end{equation*}
as soon as $p+\sigma>0$ on $\R_+^*$.
\end{proof}

\subsection{Proof of Section \ref{sec:condH}}
\begin{proof}[Proof of Proposition \ref{prop_ext_cell_line}]
Let $\eps>0$. We first focus on absorption. According to the assumptions of point \textit{i)}, there exists $x_0>0$ such that for all $x<x_0$,
$$ |H(x)|=\left|\frac{g(x)}{x}+r(x) \E[\ln\Theta]- \frac{\sigma^2(x)}{x^2}- p(x)I(x)\right| <\frac{\varepsilon}{2} |\ln x|. $$
Let us prove that there exists $a>1$ such that \ref{A1} is satisfied {\it i.e.} that there exists $a>1$ and a positive function $f$ such that
$$
H_a(x):=\frac{g(x)}{x} -a\frac{\sigma^2(x)}{x^2}- p(x)I_a(x)=f(x)+o(\ln x), \quad (x \to 0).
$$ 
We have
\begin{align}\label{eq:diffH_ext}
|H_a(x)-H(x)| \leq -r(x)\E[\ln\Theta]+(1-a)\frac{\sigma^2(x)}{x^2}+p(x)|I_a(x)-I(x)|.
\end{align}
To study the last term of \eqref{eq:diffH_ext}, let us define for all $x,z\geq 0$,
$$
f(a,x,z) =a z^2\int_0^1(1-v)(1+zx^{-1}v)^{-1-a}dv,\quad f(1,x,z) = x^2\left(zx^{-1}-\ln(1+zx^{-1})\right).
$$
From the definition of $I_a$ and $I$ in \eqref{def:Ia} and \eqref{def:I}, respectively, we have
\begin{align*}
I_a(x) = x^{-2}\int_{\R_+}f(a,x,z)\pi(dz),\quad I(x) = x^{-2}\int_{\R_+}f(1,x,z)\pi(dz).
\end{align*}
Moreover, for every $x,z>0$,
\begin{align*}
\partial_a f(a,x,z)& =  z^2 \left( \int_0^1 (1-a\ln(1+zx^{-1}v))(1+zx^{-1}v)^{-(1+a)}(1-v)dv\right)\\
& = z^2 \left( \int_0^1 \frac{1-v}{(1+zx^{-1}v)^{1+a}}dv -a\int_0^1 \frac{\ln(1+zx^{-1}v)(1-v)}{(1+zx^{-1}v)^{1+a}}dv \right)\\
& = zx \int_0^1 \frac{\ln(1+zx^{-1}v)}{(1+zx^{-1}v)^{a}}dv,\\
\end{align*}
where the last equality is obtained using integration by part in the second integral. Then, computing the integral, we get
\begin{align*}
\partial_a f(a,x,z)& = 
&= \frac{x^2}{(a-1)^2}\left[1-(a-1)\ln(1+zx^{-1})(1+zx^{-1})^{1-a}-(1+zx^{-1})^{1-a}\right]
\end{align*}

Then, according to Taylor-Lagrange's formula, there exists $ y\in (1-a,0)$ such that 
$$
1-(1+zx^{-1})^{1-a} = -\ln(1+zx^{-1})(1+zx^{-1})^{y}(1-a),
$$
and we obtain
\begin{align}\nonumber
\partial_a f(a,x,z)&=\frac{x^2}{(a-1)^2}\left[(a-1)\ln(1+zx^{-1})\left((1+zx^{-1})^{y}-(1+zx^{-1})^{1-a}\right)\right]\\\label{eq:partial_a}
&=\frac{x^2}{(a-1)}\left[(y+a-1)\ln(1+zx^{-1})^2(1+zx^{-1})^{\hat{y}}\right]
\end{align}
for some $ \hat{y}\in (1-a,y)$ according to Taylor-Lagrange's formula. Then, for $a>1$, 
$$
0\leq \partial_a f(a,x,z)\leq x^2\ln(1+zx^{-1})^2 \leq z^2
$$
and using that $\int_{\mathbb{R}_+} z^2\pi(dz)<\infty$, we obtain
\begin{align}\label{eq:majIa}
x^2\partial_a I_a(x) = \int_{\R_+}\partial_af(a,x,z)\pi(dz)\leq \int_{\mathbb{R}_+} z^2\pi(dz).
\end{align}
And, using again Taylor-Lagrange's formula, for any $x >0$ there exists $\tilde{a}(x)\in (1,a)$ such that
\begin{align*}
x^2|I_a(x)-I(x)| = x^2\left|\partial_a I_{\tilde{a}(x)}(x)\right|(a-1).
\end{align*}

Now using the previous computations, we obtain that
there exists also $x_1<x_0$ and $a_0>1$ such that for all $x<x_1$ and $1<a<a_0$,
\begin{align*}
\left|\frac{r(x)}{\ln x}\right|\left| \E[\ln\Theta]\right|<\eps/6,\quad (a-1)\frac{\sigma(x)^2}{x^2\ln x}<\eps/6,\quad \left|\frac{p(x)}{x^2\ln x}\right|\left|x^2I_{a}(x)-x^2I(x)\right|<\eps/6.
\end{align*}

Finally, combining the last inequalities with \eqref{eq:diffH_ext}, we obtain for all $x<x_1$ and $a\in (1,a_0)$,
\begin{align*}
|H_a(x)|\leq |H(x)|+ |H_a(x)-H(x)|\leq \eps |\ln x|,
\end{align*}
and thus  Condition \ref{A1} holds and we may apply Theorem \ref{thm:absorption}.\\

The proof of the second point is similar except that we have to adapt the bounds to the case $a<1$.
By assumption, there are $\eta>0$ and $x_0>0$ such that for all $x<x_0$,
$$
H(x)\leq - \ln (x^{-1})(\ln\ln (x^{-1}))^{1+\eta}.
$$
We prove that there exist $a<1$, $\eta'>0$ and $x_1>0$ such that for all $x<x_1$, 
$$
H_a(x)\leq -\ln(x^{-1})(\ln\ln(x^{-1}))^{1+\eta'}
$$
by bounding the difference between $H_a$ and $H$ on $[0,x_0]$. Note that for $a<1$, similarly to \eqref{eq:partial_a}, there exist $y\in (0, 1-a)$ and $\hat{y}\in (y, 1-a)$ such that
\begin{align}\nonumber\label{majz2a<1}
\partial_a f(a,x,z)&=\frac{x^2}{(1-a)}\left[(1-a-y)\ln(1+zx^{-1})^2(1+zx^{-1})^{\hat{y}}\right]\\
&\leq x^2 \ln(1+zx^{-1})^2(1+zx^{-1})\leq z^2,
\end{align}
so that we can conclude as before.\\

We now turn to the proof of results on the explosion of the process. 

 According to the assumptions of point \textit{i)}, there exists $x_0>0$ such that for all $x>x_0$,
$$ |H(x)|=\left|\frac{g(x)}{x}+r(x) \E[\ln\Theta]- \frac{\sigma^2(x)}{x^2}- p(x)I(x)\right| <\frac{\varepsilon}{2} \ln x. $$ 
We now prove that there exists $a<1$ such that \ref{SNinfty} is satisfied {\it i.e.} that there exist $a<1$ and a positive function $f$ such that
$$
\tilde{H}_a(x):=\frac{g(x)}{x}-r(x) \frac{1-\E[\Theta^{1-a}]}{1-a}- a\frac{\sigma^2(x)}{x^2}- p(x)I_a(x)=-f(x)+o(\ln x).
$$
We have
\begin{align}\label{eq:diffH}
|\tilde{H}_a(x)-H(x)| \leq  r(x)\left|\frac{1-\E[\Theta^{1-a}]}{1-a}+\E[\ln\Theta]\right|+(1-a)\frac{\sigma^2(x)}{x^2}+p(x)|I_a(x)-I(x)|.
\end{align}
For the first term of \eqref{eq:diffH}, let us consider for $\theta \in (0,1)$ the functions 
 $$f_\theta : y \in \R \mapsto \theta^y \quad \text{and} \quad g_\theta : y \in \R \mapsto \frac{1- \theta^y}{y} .$$
 Then
 $$ f'_\theta (y)=(\ln \theta) \theta^y , \quad f''_\theta (y)=(\ln \theta)^2 \theta^y \geq 0 \quad \text{and} \quad g'_\theta(y) = \frac{1}{y^2}(f_\theta(y)-f_\theta(0)-yf_\theta'(y)) . $$
Using Taylor's formula twice, we get the existence of $0 \leq \lambda(\theta,y),\mu(\theta,y)\leq 1$ such that
\begin{align*}
 g'_\theta(y) = \frac{1}{y}(f'_\theta(\lambda (\theta,y) y)-f_\theta'(y)) = (\lambda (\theta,y)-1)f''_\theta(\mu (\theta,y)y)\leq 0.
\end{align*}
We deduce that the function 
$$ h: y \in \R \mapsto \int_0^1 \frac{1- \theta^y}{y} \kappa(d\theta)=\int_0^1 \frac{f_\theta(0)- f_\theta(y)}{y} \kappa(d\theta)=-\int_0^1 f'_\theta(\lambda (\theta,y) y) \kappa(d\theta) $$
is non-increasing. Moreover, 
$$ - f'_\theta(\lambda (\theta,y) y) = - (\ln \theta)\theta^{\lambda (\theta,y) y} \xrightarrow[y \to 0]{}-\ln \theta .$$
As Condition \eqref{cond_lntheta} holds, we deduce by monotone convergence that 
$$ \lim_{y \to 0}h(y)= \lim_{a \to 1}\int_0^1 \frac{1- \theta^{1-a}}{1-a}\kappa (d\theta)= - \int_0^1 \ln \theta \kappa (d\theta)= -\E[\ln \Theta]. $$

For the last term of \eqref{eq:diffH}, we can use again \eqref{eq:majIa}, which is satisfied in this case according to \eqref{majz2a<1}. Now combining the previous computations, we obtain that
there exists $x_1>x_0$ and $a_0\in (0,1)$ such that for all $x>x_1$ and $a_0<a<1$,
$$
\left|\frac{r(x)}{\ln x}\right|\left|h(1-a_0)-h(0)\right|<\eps/6,\quad\left|\frac{p(x)}{x^2\ln x}\right|\left|x^2I_{a}(x)-x^2I(x)\right|<\eps/6,\quad (1-a)\frac{\sigma(x)^2}{x^2\ln x}<\eps/6. $$

Finally, combining the last inequalities with \eqref{eq:diffH}, we obtain for all $x>x_1$ and $a\in (a_0,1)$,
\begin{align*}
|\tilde{H}_a(x)|\leq |H(x)|+ |\tilde{H}_a(x)-H(x)|\leq \eps \ln(x)
\end{align*}
and thus  Condition \ref{SNinfty} holds and we may apply Theorem \ref{thm:explosion}.

The proof for the case $a>1$ is similar.
\end{proof}

\subsection{Proofs of Section \ref{section_class}}

We now turn to the proof of Theorem \ref{th:convXt}. Let $t_0>0$ be fixed.
First, we prove that if the division mechanism of the cells and the random fluctuations are stronger than the growth of the parasites in the sense of 
\ref{B0} for some $x_0> 0$, then the stopping times $T_i(x_0)$, are finite a.s. for all $i\geq 0$ where $T_0 = 0$ and for all $i\geq 1$,
\begin{equation}\label{eq:Ti}
T_i(x_0) = \inf\lbrace t\geq T_{i-1}(x_0)+t_0,\ X_t\leq x_0\rbrace.
\end{equation}

\begin{lemma}\label{lem:Tifini}
Under Assumptions \ref{ass_PSS} and Condition \ref{SNinfty}, if \ref{B0} holds  for some $\eta>0$ and $x_0>0$, we have
$\E[T_i(x_0)]<\infty$ for all $i\geq 0$. 
\end{lemma}

\begin{proof}[Proof of Lemma \ref{lem:Tifini}]
Let us consider $\tau = \tau^-(x_0)\wedge \tau^+(x_1)$ where $x_1\geq x_0$ and the $\tau^{\pm}$'s have been defined in \eqref{tau1}.
According to the strong Markov property, we only have to prove that $\E_x(\tau^-(x_0))<\infty$ for all $x\geq 0$.
By It\^o's formula, we have for all $t\geq 0$
\begin{align*}
\ln(X_{t\wedge \tau}) = &  \ln(X_{0})+\int_{0}^{t\wedge \tau}\frac{g(X_s)}{X_s}ds-\int_{0}^{t\wedge \tau}\frac{\sigma^2(X_s)}{X_s^2}ds 
+\E\left[\ln \Theta \right]\int_{0}^{t\wedge \tau} r(X_s)ds \nonumber \\
&+\int_{0}^{t\wedge \tau} p(X_s)\int_0^\infty \left[\ln(X_{s}+z)-\ln(X_{s})-z/X_{s}\right]\pi(dz)ds +M_{t\wedge \tau},
\end{align*}
where $(M_{s\wedge \tau}, s \geq 0)$ is a martingale with null expectation.
Then, using Condition \ref{B0}, we obtain
\begin{align}\label{eq:itoln}
\ln(X_{t\wedge \tau}) - \ln(X_{0})
 \leq - \eta  (t \wedge \tau) + M_{t\wedge \tau}.
\end{align}
Notice that $X_{\tau^-(x_0)}$ may be equal to $x_0$ if there is no jump at time $\tau^-(x_0)$, or equal to $X_{\tau^-(x_0)^-} \Theta$ where 
$X_{\tau^-(x_0)^-}\geq x_0$, $\Theta$ is independent of $X_{\tau^-(x_0)^-}$ and distributed according to $\kappa$. As a consequence, for all $t\geq 0$
$$ \ln\left(X_{t\wedge \tau}\right) \geq 
 \ln\left(\Theta x_0\right),$$
almost surely. Then, taking the expectation in \eqref{eq:itoln}, using the last inequality and letting $t$ tend to infinity yield for all $x>0$
$$
\E_x(\tau)\leq \frac{1}{\eta}\ln\left(\frac{x}{\Theta x_0}\right).
$$
According to Theorem \ref{thm:explosion}, Condition \ref{SNinfty} yields that for all $x>0$, $\P_x(\tau^+(\infty)<\infty)=0$, so that $\liminf_{x_1\rightarrow +\infty}\tau^-(x_0)\wedge \tau^+(x_1)=\tau^-(x_0).$ Next, by Fatou's Lemma, 
\begin{align*}
\E_x[\tau^-(x_0)] = \E_x\left[\liminf_{x_1\rightarrow +\infty}\tau^-(x_0)\wedge \tau^+(x_1)\right]&\leq \liminf_{x_1\rightarrow +\infty}\E_x\left[\tau^-(x_0)\wedge \tau^+(x_1)\right]\\
& \leq \frac{1}{\eta}\ln\left(\frac{x}{\Theta x_0}\right)<\infty,
\end{align*}
which ends the proof.
\end{proof}

Let $t_1>0$.
Similarly, if the growth of the parasites is stronger than the division mechanism of the cells and the random fluctuations in the sense of 
\ref{LFG} for some $x_1> 0$, then the stopping times $\tT_i(x_1)$, are finite a.s. for all $i\geq 0$ where $\tT_0 = 0$ and for all $i\geq 1$,
\begin{equation*}\label{eq:tTi}
\tT_i(x_1) = \inf\lbrace t\geq \tT_{i-1}(x_1)+t_1,\ X_t\geq x_1\rbrace.
\end{equation*}

\begin{lemma}\label{lem:Tifini2}
Under Assumptions \ref{ass_PSS} and Condition \ref{A1}, if $x_1> 0$ is such that \ref{LFG} is satisfied for some $\eta>0$, and if $\int_{\mathbb{R}_+}\ln(1+z)\pi(dz)<\infty$, we have
$\E[\tT_i(x_1)]<\infty$ for all $i\geq 0$. 
\end{lemma}

\begin{proof}[Proof of Lemma \ref{lem:Tifini2}]
Without loss of generality, we assume that $x_1>1$. Following the same lines as in the proof of Lemma \ref{lem:Tifini}, we obtain
\begin{align}\label{eq:itoln2}
\ln(X_{t\wedge \tau}) - \ln(X_{0})
 \geq \eta  (t \wedge \tau) + M_{t\wedge \tau},
\end{align}
where $\tau = \tau^-(x_0)\wedge \tau^+(x_1)$ where $x_1\geq x_0$.
As in the proof of Lemma \ref{lem:Tifini}, considering both cases of $X$ exceeding $x_1$ thanks to a jump or not, we obtain  for all $t\geq 0$
$$ \E_x\left[\ln\left(X_{t\wedge \tau}\right)\right] \leq 
 \int_{\mathbb{R}_+}\ln\left(x_1+z\right)\pi(dz).$$
Then, taking the expectation in \eqref{eq:itoln2}, using the last inequality and letting $t$ tend to infinity yield for all $x>0$
$$
\E_x(\tau)\leq \frac{1}{\eta}\int_{\mathbb{R}_+}\ln\left(x_1+z\right)\pi(dz)<\infty.
$$
According to Theorem \ref{thm:absorption}, Condition \ref{A1} yields that for all $x>0$, $\P_x(\tau^-(0)<\infty)=0$, so that $\liminf_{x_0\rightarrow 0}\tau^-(x_0)\wedge \tau^+(x_1)=\tau^+(x_1),$ and we conclude by Fatou's Lemma as before. 

\end{proof}

\begin{proof}[Proof of Theorem \ref{th:convXt}]

Apart from the proof of Equation \eqref{extin_B0A2}, the proof of Theorem \ref{th:convXt}
follows directly from Lemmas \ref{lem:Tifini} and \ref{lem:Tifini2}, and \cite[Theorem 7.1.4]{BN}.
It is very similar to the proof of the second point of \cite[Theorem 1]{hermann2018markov} for instance and we refer the 
reader to this paper for details. 
Equality \eqref{esp_Xinfty} is obtained by taking expectation in \eqref{eq:EDS}.

Let us now prove \eqref{extin_B0A2}. 
To do this, we first show that there exist $y_0,t_0$ and $\alpha>0$ such that,
\begin{equation} \label{probaextunif} \inf_{0\leq x \leq y_0}\P_x(X_{t_0}=0)\geq  \alpha. \end{equation}
Let us fix $a<1$ such that \ref{A2} is satisfied and $\delta<(3-2a)^{-1}$. On page \pageref{pageineq}, we have proved that there exist two non-negative functions on $\R_+$, $\mathfrak{t}$ and $\mathfrak{p}$ such that for all $\eps<e^{-1/(1-\delta)}$ such that \ref{A2} is satisfied for $x\leq \eps$, and $z \in (\eps^{1+\delta},\eps^{1-\delta})$, 
\begin{equation} \P_z\left(\tau^-(0)\leq \mathfrak{t}(\eps)\right)\geq \mathfrak{p}(\eps). \label{ineq_magique}
\end{equation}
where 
$$ \mathfrak{t}(\eps)= \sum_{n=0}^\infty t(\varepsilon^{(1+\delta)^n}),\quad \mathfrak{p}(\eps)= e^{-2\left( \E[\Theta^{1-a}]^{-1} +1\right)\varepsilon^{(1-a)\delta}(1-\varepsilon^{(1-a)\delta^2})^{-1}},  $$
$t(\cdot)$ has been defined in \eqref{teps}. By a classical functional study, we can check that the function $\mathfrak{t}$ is non-decreasing and the function $\mathfrak{p}$ is non-increasing.

Let us take $\eps>0$ such that \eqref{ineq_magique} is satisfied and $z\leq\eps^{1-\delta}$. Then, there exists $\eps_1 \leq \eps$ such that $\eps_1^{1+\delta}<z<\eps_1^{1-\delta}$. Now, by monotonicity, we get:
\begin{align*}
\P_{z}\left(\tau^-(0)\leq \mathfrak{t}(\eps)\right) \geq \P_{z}\left(\tau^-(0)\leq \mathfrak{t}(\eps_1)\right)\geq \mathfrak{p}(\eps_1) \geq \mathfrak{p}(\eps).
\end{align*}
Equation \eqref{probaextunif} is thus proven, if we take $y_0=\eps^{1-\delta}$, $t_0=\mathfrak{t}(\eps)$ and $\alpha=\mathfrak{p}(\eps)$.
Next, we need to show that there exist $t_1,\alpha>0$ such that
$$
\inf_{\eps^{1-\delta}\leq x\leq x_0}\P_x(X_{t_1}\leq \eps^{1-\delta})\geq \alpha>0,
$$
where $x_0>0$ is such that \ref{B0} is satisfied. We obtain this property by following the proof of \eqref{temps_descente}.

Recall the definition of $T_i(x_0)$ in \eqref{eq:Ti}. By the strong Markov property and \eqref{probaextunif}, we get for all $x\geq 0$ and all $i\geq 0$,
\begin{align*}
\mathbb{P}_{x}\left(X_{T_i(x_0)+t_0} =0 \big| (X_t,\ t\leq T_i(x_0)),\ T_i(x_0)<\infty\right)\geq \alpha.
\end{align*}
Applying Lemma \ref{lem:Tifini} and the strong Markov property,
we deduce that for any $x\geq 0$,
\begin{align*}
\P_x \left( X_t >0, \forall \ t\geq 0 \right)& \leq
\P_x\left(\forall i\geq 0,\ X_{T_i(x_0)+t_0} >0\right) 
= \P_x\left(\forall i\geq 0,\ X_{T_i(x_0)+t_0} >0,\ T_i(x_0)<\infty\right) = 0.
\end{align*}
This concludes the proof of the second point.\\

Finally, we prove \eqref{expl_A1LFG}. The proof is very similar to the one of \eqref{extin_B0A2} and we will not give all the details. Following the proof of \cite[Theorem 2.8]{li2017general} with the only difference that we choose the function $t$ as defined in \eqref{teps}, we obtain the existence of $a>1$ and of a small positive $\delta$ such that for every small enough $\eps$, 
$$ \P_{1/\eps} \left( \tau^+(\infty)\leq \mathfrak{t}(\eps) \right)\geq \prod_{k=1}^\infty \left( 1- 2 \eps^{(1-a)\delta (1+\delta)^n} \right) .$$
We end the proof as in the case of absorption.
\end{proof}

We now prove Proposition \ref{3_behaviours} and Corollary \ref{iiimoreprecise} which concern the absorption of the process.

\begin{proof}[Proof of Proposition \ref{3_behaviours}]
Let us introduce the following time change:
\begin{align} \label{time_change}
X_t = Y_{\int_0^t r(X_s)ds}.
\end{align}
According to Theorem 1.4 in Section 6 in \cite{EK}, there is a version of $X$ satisfying \eqref{time_change} for a process $Y$ that is a solution of the 
martingale problem with associated generator 

\begin{align*}
\mathcal{G}_Yf(x) = & \frac{g(x)}{r(x)}f'(x)+\frac{\sigma^2(x)}{r(x)}f''(x)+\int_0^1 (f(\theta x)-f(x))\kappa(d\theta)\\
& +\frac{p(x)}{r(x)}\int (f(x+z)-f(x)-zf'(x))\pi(dz),
\end{align*}
and is a weak solution to
\begin{multline}
\label{EDS_Y}
Y_t = Y_0+\int_0^t \frac{g(Y_s)}{r(Y_s)}ds+\int_0^t\sqrt{\frac{2\sigma^2(Y_s)}{r(Y_s)}}dB_s+\int_0^t\int_0^1\int_0^1 
(\theta - 1)Y_{s^-} N(ds,dx,d\theta)\\
 +\int_0^t\int_0^{p(Y_{s^-})/r(Y_{s^-})}\int_{\mathbb{R}_+}z\widetilde{Q}(ds,dx,dz),
\end{multline}
where we chose on purpose the same Poisson Point measures as in the definition of $X$ in \eqref{eq:EDS}.
In fact, as \eqref{EDS_Y} admits a unique strong solution
(see the proof of Proposition \ref{prop_sol_SDE}),
$Y$ is even pathwise unique.
Now let us introduce the processes $(K_t, t \geq 0)$ and $(Z_t, t \geq 0)$ via
$$ K_t:= \int_0^t \frac{g(Y_s)}{Y_sr(Y_s)} ds +  \int_0^t \int_0^{1}\int_0^1 \ln \theta N(ds,dx,d\theta) $$
and
$$ Z_t := Y_t e^{-K_t}. $$
Then an application of It\^o's formula with jumps gives
$$ Z_t = Y_0 + \int_0^t e^{-K_s} \sqrt{\frac{2\sigma^2(Y_s)}{r(Y_s)}}dB_s+ 
\int_0^t \int_0^{p(Y_{s^-})/r(Y_{s^-})}\int_{\R_+}e^{-K_{s^-}}z \widetilde{Q}(ds,dx,dz). $$
Hence $(Z_t, t \geq 0)$ is a non-negative local martingale. 
In particular it is a non-negative supermartingale and there exists a finite random variable $W$ such that 
\begin{equation} \label{eq:conv_W}
  \lim_{t\rightarrow +\infty}Y_t e^{-K_t}=W, \quad \text{a.s.}
\end{equation}
Under the assumptions of point {\it i)}, $K$ is smaller than a L\'evy process with drift $-\eta$. 
As a consequence, $e^{-K_t}$ goes to $+\infty$,
and we deduce from \eqref{eq:conv_W} that $Y$ goes to $0$.
As by assumption $\int_0^t r(X_s)ds\geq \underline{r}t$, we deduce from the time change \eqref{time_change} that $X$ goes to $0$.\\

We turn to the proof of {\it ii)} and consider the associated assumptions. In this case, $K$ is smaller than an oscillating 
L\'evy process, and we have $\liminf_{t \to \infty} K_t=-\infty$. This implies 
$\liminf_{t \to \infty} Y_t=0$.
Again, 
we deduce from the time change \eqref{time_change} that 
$\liminf_{t \to \infty} X_t=0$.\\

Let us now prove {\it iii)} as well as point {\it iii)} of Corollary \ref{iiimoreprecise}.
We use arguments similar to the ones needed to prove \cite[Corollary 2]{BPS}.
As we are in a more general setting, we need to adapt several of these arguments.
Most adaptations are obtained by couplings with well-chosen processes.

We denote by $M$ a finite bound of the function $x \mapsto (\sigma^2(x)+p(x))/(xr(x))$.
The first step consists in showing that $\P(W>0|K)>0$. To this aim, we look for a function 
$\tilde{v}_t(s,\lambda,K,Y)$, differentiable with respect
to the variable $s$, such that $F(s,Z_s)$ is a martingale conditional on $K=(K_s, s \geq 0)$, where
$$F(s,x):= \exp \{-x\tilde{v}_t(s,\lambda,K,Y)\}. $$
By an application of It\^o's formula with jumps, we obtain that 
$\tilde{v}_t$ has to satisfy for every $s \leq t$,
\begin{equation} \label{def_v3} \frac{\partial}{\partial s}\tilde{v}_t(s,\lambda,K,Y)
=e^{K_s}\tilde{\psi}_0\left( \tilde{v}_t(s,\lambda,K,Y)e^{-K_s},Y_s \right), \quad 
\tilde{v}_t(t,\lambda,K,Y)=\lambda, \end{equation}
where 
\begin{equation}\label{defpsitilde0} 
\tilde{\psi}_0(\phi,x)= \frac{\sigma^2(x)}{xr(x)}\phi^2 + \frac{p(x)}{xr(x)} \int_0^\infty \left( e^{-\phi z}-1 +\phi z \right)\pi(dz). 
\end{equation}
In particular
\begin{equation} \label{laplace} \E_y \left[ e^{-\lambda Z_t}\Big| K \right] = e^{-y \tilde{v}_t(0,\lambda,K,Y)}.
\end{equation}

Let us now introduce a function $v_t(s,\lambda,K)$, differentiable with respect
to the variable $s$, and satisfying.
$$ \frac{\partial}{\partial s}v_t(s,\lambda,K)
=e^{K_s}\psi_0\left( v_t(s,\lambda,K)e^{-K_s}\right), \quad v_t(t,\lambda,K)=\lambda, $$
where 
$$ \psi_0(\phi)= M \left( \phi^2 + \int_0^\infty \left( e^{-\phi z}-1 +\phi z \right)\pi(dz)\right). $$
Then for every $\lambda,x\geq 0$, $$\tilde{\psi}_0(\lambda,x)\leq \psi_0(\lambda)$$
and as a consequence, for all $s \leq t$, $\lambda > 0$
$$ v_t(s,\lambda,K)\leq \tilde{v}_t(s,\lambda,K,Y). $$
Combining this last inequality with \eqref{laplace}, we obtain that
\begin{equation*} 
\E_y \left[ e^{-\lambda Z_t}\Big| K \right] \leq e^{-y v_t(0,\lambda,K)}.
\end{equation*}
Taking $\lambda=1$ and letting $t$ go to infinity we get
\begin{equation*} \E_y \left[ e^{-W}\Big| K \right] \leq e^{-y v_\infty(0,1,K)}<1, \end{equation*}
where the last inequality comes from \cite{BPS} (see the proof of Corollary 2 on page 7). This allows us to conclude that
\begin{equation} \label{W_pos} \P(W>0|K)>0. \end{equation}
Under the assumptions of point {\it iii)} $K$ is larger than a L\'evy process with drift $\eta$ and as a consequence, $e^{-K_t}$ goes to $0$. From \eqref{eq:conv_W} and the previous inequality, we deduce that 
$$ \liminf_{t \to \infty} Y_t = \infty $$
with positive probability. In particular, this implies that 
$$ \liminf_{t \to \infty}X_t = \liminf_{t \to \infty} Y_{\int_0^t r(X_s)ds} \geq \liminf_{t \to \infty} Y_t >0 $$
with positive probability.\end{proof}

\begin{proof}[Proof of Corollary \ref{iiimoreprecise}]
Let us begin with point {\it iii)}. We showed in the proof of Proposition \ref{3_behaviours} (see \eqref{eq:conv_W}) that under the assumptions of point {\it iii)},
\begin{equation} \label{eq:conv_W2}
  \lim_{t\rightarrow +\infty}Y_t e^{-K_t}=W, \quad \text{a.s.},
\end{equation}
where $K$ is larger than a L\'evy process with drift $\eta$. Moreover, as $r(X_s)\geq \uline{r}$ for any $s\geq 0$, we have for any $t\geq 0$,
\begin{equation} \label{def_rhot}
\rho(t):= \int_0^t r(X_s)ds \geq \uline{r}t.
\end{equation}
Combining \eqref{time_change}, \eqref{eq:conv_W2} and \eqref{def_rhot}, we obtain
$$ \lim_{t \to \infty} Y_{\int_0^t r(X_s)ds}e^{-K_{\int_0^t r(X_s)ds}}
= \lim_{t \to \infty} X_t e^{-K_{\rho(t)}}=W \quad \text{a.s.} $$
Finally, \eqref{W_pos} allows to conclude the proof of \eqref{ineqW}.\\

Let us now focus on points {\it i)} and {\it ii)}.
The idea of the proof is to compare the survival probability of $X$ with the survival probability of 
a Feller diffusion with jumps, whose asymptotic behaviour has been studied in \cite{BPS}.

Let us recall the definitions of $\tilde{v}$ and $\tilde{\psi}_0$ in \eqref{def_v3} and \eqref{defpsitilde0}, respectively.
Now, according to \ref{B1}, $\inf_{x\geq 0}r(x)>0$ so that by assumption, $\inf_{x \geq 0}\sigma^2(x)/(xr(x))>0$. Therefore, there exists $\mathfrak{a}>0$ 
such that for every $\phi,x\geq 0$, $$\tilde{\psi}_0(\phi,x)\geq \mathfrak{a} \phi^2.$$
Hence, if we introduce $\bar{v}$ as the solution to
\begin{equation} \label{defvi} \frac{\partial}{\partial s}\bar{v}_t(s,\lambda,K)
= \mathfrak{a} e^{-K_s}\left( \bar{v}_t(s,\lambda,K)\right)^2, \quad \bar{v}_t(t,\lambda,K)=\lambda, \end{equation}
we obtain that for all $s \leq t$, $\lambda > 0$,
$$ \bar{v}_t(s,\lambda,K)\geq \tilde{v}_t(s,\lambda,K,Y), $$
implying, using \eqref{laplace},
\begin{equation*} \E_y \left[ e^{-\lambda Z_t}\Big| K \right] \geq e^{-y \bar{v}_t(0,\lambda,K)}. \end{equation*}
Letting $\lambda$ go to infinity yields
\begin{equation*} \P_y \left( Y_t=0| K \right) = \P_y \left( Z_t=0| K \right) 
\geq e^{-y \bar{v}_t(0,\infty,K)}.
\end{equation*}
But \eqref{defvi} has an explicit solution, and
$$ \bar{v}_t(0,\infty,K) = \left( \mathfrak{a} \int_0^t e^{-K_u}du \right)^{-1}.$$
We thus deduce that for any $t \geq 0$,
$$ \P_y \left( Y_t>0\right) \leq 1-\E \left[ e^{-y( \mathfrak{a} \int_0^t e^{-K_u}du )^{-1}} \right]. $$
A direct application of \cite[Theorem 7]{BPS} with $F(x) = 1-e^{-y( \mathfrak{a} x)^{-1}}$ gives the long time behaviour of the right hand side of the previous inequality. Finally,
\begin{align*} \P_y \left( X_t>0\right)&=\P_y \left( Y_{\int_0^t r(X_s)ds} >0\right)
\leq \P_y \left( Y_{\underline{r}t} >0\right) \leq 1- \E \left[ e^{-y( \mathfrak{a} \int_0^{\underline{r}t} e^{-K_s}ds )^{-1}} \right]. 
\end{align*}
\end{proof}

We end this proof section by the study of the conditions under which the process $X$ or its superior limit drift to infinity.

\begin{proof}[Proof of Equation \eqref{rem_no_abs}] \label{no_abs}
Let $a\in\mathcal{A}$. Then $(1,a]\subset\mathcal{A}$ by definition. In particular, we may take $a \in (1,2]\bigcap \mathcal{A}$, which implies that
$$\frac{g(x)}{x}-a\frac{\sigma^2(x)}{x^2}\geq 
\frac{g(x)}{x} - \frac{2\sigma^2(x)}{x^2}.
 $$
Now, recall that for $x>0$,
\begin{equation*}
 I(x)=\lim_{a\rightarrow 1}I_a(x) = \int_{\mathbb{R}_+} \left[zx^{-1}+\ln\left(\frac{1}{1+zx^{-1}}\right)\right]\pi(dz).
\end{equation*}
Using that for any $u>0$, $\ln(u)<u-1$, we obtain that 
$$ I(x)<\int_{\R_+} \left( \frac{1}{1+zx^{-1}}-1+zx^{-1} \right)\pi(dz)=\int_{\R_+} \frac{z^2x^{-2}}{1+zx^{-1}}\pi(dz).$$
And by continuity, we deduce that there exists $a \in \mathcal{A}$ such that
$$ I_a(x)<\int_{\R_+} \frac{z^2x^{-2}}{1+zx^{-1}}\pi(dz), $$
hence
$$ -p(x)I_a(x)\geq -
  p(x)\int_{\R_+} \frac{z^2x^{-2}}{1+zx^{-1}}\pi(dz).$$
This ends the proof.
\end{proof}

\begin{proof}[Proof of Proposition \ref{3_behaviours_expl}]
Recall that the time-changed process $Y$ is a weak solution to \eqref{EDS_Y}, and let us introduce the process $V$ via $V_t:=1/Y_t , t \geq 0$, which is well defined as $Y_t$ does not reach $0$ under Assumption \ref{B4} (see the third point of Remark \ref{rem:cond-slowfast}). Applying Itô's formula with jumps, we obtain that $V$ is a weak solution to the SDE:
\begin{multline*}
V_t = V_0 +  \int_0^t V_s\left( 2V_s^2 \frac{\sigma^2(V_s^{-1})}{r(V_s^{-1})}-V_s\frac{g(V_s^{-1})}{r(V_s^{-1})} + \frac{p(V_s^{-1})}{r(V_s^{-1})} \int_{\R_+}\left( \frac{1}{1+zV_s}-1+zV_s \right)\pi(dz) \right)ds\\-\int_0^tV_s^2\sqrt{\frac{2\sigma^2(V_s^{-1})}{r(V_s^{-1})}}dB_s+\int_0^t\int_0^1\int_0^1 
\left(\frac{1}{\theta} - 1\right)V_{s^-} N(ds,dx,d\theta)\\
 +\int_0^t\int_0^{p(V^{-1}_{s^-})/r(V^{-1}_{s^-})}\int_{\mathbb{R}_+}V_{s^-}\left(\frac{1}{1+zV_{s^-}}-1\right)\widetilde{Q}(ds,dx,dz).
\end{multline*}
Now, if we introduce the processes $\tilde{K}$ and $\tilde{Z}$ via:
\begin{multline*}
\tilde{K}_t =  \int_0^t \left( 2V_s^2 \frac{\sigma^2(V_s^{-1})}{r(V_s^{-1})}-V_s\frac{g(V_s^{-1})}{r(V_s^{-1})} + \frac{p(V_s^{-1})}{r(V_s^{-1})} \int_{\R_+}\left( \frac{1}{1+zV_s}-1+zV_s \right)\pi(dz) \right)ds\\-\int_0^t\int_0^1\int_0^1 
\ln \theta N(ds,dx,d\theta)
\end{multline*}
and $\tilde{Z}_t:=V_te^{-\tilde{K}_t}$ for any $t\geq 0$, we obtain, applying again Itô formula with jumps:
\begin{multline*}
\tilde{Z}_t = V_0 -\int_0^te^{-\tilde{K}_s}V_s^2\sqrt{\frac{2\sigma^2(V_s^{-1})}{r(V_s^{-1})}}dB_s\\
 +\int_0^t\int_0^{p(V^{-1}_{s^-})/r(V^{-1}_{s^-})}\int_{\mathbb{R}_+}e^{-\tilde{K}_{s^-}}V_{s^-}\left(\frac{1}{1+zV_{s^-}}-1\right)\widetilde{Q}(ds,dx,dz).
\end{multline*}
In other words, $\tilde{Z}$ is a non-negative local martingale, and thus a supermartingale. It converges to a non-degenerated and non-negative random variable $\tilde{W}$. We conclude the proof as the proof of points $i)$ and $ii)$ of Proposition \ref{3_behaviours}.
\end{proof}

\appendix
\section{Detailed computation of the limit of $I_a$}\label{app:I}
In this section, we prove that for all $x>0$,
$$
\lim_{a\rightarrow 1}I_a(x)= I(x)
$$
where we recall that
$$
I_a(x) = a\int_{\mathbb{R}_+} \frac{z^2}{x^2}\left(\int_0^1 \frac{1-v}{(1+zx^{-1}v)^{1+a}}dv\right)\pi(dz),\quad I(x)= -\int_{\mathbb{R}_+} \left[\ln\left(1+zx^{-1}\right)-zx^{-1}\right]\pi(dz).
$$
Note that computing the integral, we have
\begin{align}\label{eq:Ia2}
I_a(x) = \int_{\mathbb{R}_+}\left(zx^{-1}+\frac{1-(1+zx^{-1})^{1-a}}{1-a}\right)\pi(dz),
\end{align}
and
$$
\lim_{a\rightarrow 1}\frac{1-(1+zx^{-1})^{1-a}}{1-a} = -\ln(1+zx^{-1}).
$$
By Taylor's formula applied to the function $y\mapsto (1+y)^{1-a}$, there exists $\zeta\in (0, zx^{-1})$ such that
$$
I_a(x) = \int_{\R_+}zx^{-1}\left(1-(1+\zeta)^{-a}\right)\pi(dz)\leq  \int_{\R_+}zx^{-1}\left(1-(1+zx^{-1})^{-1}\right)\pi(dz)<\infty,
$$
according to Assumption \ref{ass_PSS}. Then, using the dominated convergence theorem, we obtain
$$
\lim_{a\rightarrow 1}I_a(x)=I(x). 
$$
\section*{Acknowledgments}
The authors are grateful to V. Bansaye for his advice and comments and to B. Cloez for fruitful discussions.
They also want to thank the two anonymous
referees for several corrections and improvements.
This work was partially funded by the Chair "Mod\'elisation Math\'ematique et Biodiversit\'e" of VEOLIA-Ecole Polytechnique-MNHN-F.X. and by the French national research agency (ANR) via project MEMIP (ANR-16-CE33-0018) and project NOLO (ANR-20-CE40-0015).

\bibliographystyle{abbrv}
\bibliography{biblio}

\end{document}